\documentclass[11pt,titlepage]{amsart}

\usepackage{etex}
\usepackage{latexsym}
\usepackage{amssymb}
\usepackage{amsmath}
\usepackage{amsfonts}
\usepackage[mathscr]{eucal}
\usepackage{MnSymbol}
\usepackage{color}

\usepackage{amsaddr}
\makeatletter
\renewcommand{\email}[2][]{%
  \ifx\emails\@empty\relax\else{\g@addto@macro\emails{,\space}}\fi%
  \@ifnotempty{#1}{\g@addto@macro\emails{\textrm{(#1)}\space}}%
  \g@addto@macro\emails{#2}%
}
\makeatother

\usepackage{amsrefs}

\usepackage{a4wide}
\usepackage{amscd}

\usepackage{graphicx}
\usepackage{tikz}
	\usetikzlibrary{shapes,arrows,positioning,decorations.pathreplacing,calc}
\usepackage{stackrel}

\usepackage{float} 

\usepackage{amsthm}
\usepackage[all,cmtip]{xy}
\usepackage[pdfstartview=FitH]{hyperref}
\usepackage{pst-grad}
\usepackage{comment}
\usepackage{psfrag}
\usepackage{verbatim}
\usepackage{hyperref}
\hypersetup{colorlinks,allcolors=blue}
\usepackage{multirow}
\usepackage{caption} 
\usepackage{tabularx}
\usepackage{enumerate}
\usepackage{caption} 
\usepackage{soul}

\usepackage{lineno}

\makeatletter
\DeclareFontFamily{U}{tipa}{}
\DeclareFontShape{U}{tipa}{m}{n}{<->tipa10}{}
\newcommand{\arc@char}{{\usefont{U}{tipa}{m}{n}\symbol{62}}}%

\newcommand{\arc}[1]{\mathpalette\arc@arc{#1}}

\newcommand{\arc@arc}[2]{%
  \sbox0{$\m@th#1#2$}%
  \vbox{
    \hbox{\resizebox{\wd0}{\height}{\arc@char}}
    \nointerlineskip
    \box0
  }%
}
\makeatother

\newtheorem{thm}{Theorem}[section]
\newtheorem{cor}[thm]{Corollary}
\newtheorem{lem}[thm]{Lemma}
\newtheorem{prop}[thm]{Proposition}
\newtheorem{theorem}{Theorem}

\newtheorem{corollary}[theorem]{Corollary}

\theoremstyle{definition}
\newtheorem{definition}[thm]{Definition}

\newtheorem{example}[thm]{Example}

\newtheorem{rem}[thm]{Remark}

\numberwithin{equation}{section}

\usepackage[labelformat=simple]{subcaption}

\DeclareMathOperator{\susp}{Susp}
\DeclareMathOperator{\curv}{curv}
\DeclareMathOperator{\dime}{dim}

\newcommand{\RR}{\mathbb{R}}
\newcommand{\Z}{\mathbb{Z}}
\newcommand{\CC}{\mathbb{C}}

\newcommand{\QQ}{\mathbb{Q}}

\newcommand{\Sing}{\mathit{Sing}}

\newcommand{\RP}{\mathbb{R}P}

\newcommand{\soso}{\mathrm{T^{2}}}
\newcommand{\Circ}{\mathrm{T^{1}}}

\newcommand{\Ss}{\mathrm{S}}

\newcommand{\ZZ}{\mathbb{Z}}

\begin{document}


\title[Torus actions on Alexandrov $4$-spaces]{Torus actions on Alexandrov $4$-spaces}


\author[D.~Corro]{Diego Corro$^\ast$}
\address[D.~Corro]
{Universität zu Köln, Germany and\\ Instituto de Matem\'{a}ticas, Universidad Nacional Aut\'{o}noma de M\'{e}xico (UNAM), Oaxaca de Ju\'{a}rez, M\'{e}xico.}

\email[D.~Corro]{\href{mailto:diego.corro@im.unam.mx}{diego.corro.math@gmail.com}}
\thanks{$^\ast$Supported by  CONACyT-DAAD Doctoral scholarship no.~409912, DFG RTG-2229 ``Asymptotic Invariants and Limits of Groups and Spaces", DGAPA-UNAM Postdoctoral Fellowship of the Institute of Mathematics, DFG-Eigenestelle Fellowship CO 2359/1-1.}

\author[J.~N\'u\~nez-Zimbr\'on]{Jes\'us N\'u\~nez-Zimbr\'on$^{\ast\ast}$}
\address[J.~N\'u\~nez-Zimbr\'on]
{Facultad de Matemáticas, Universidad Nacional Aut\'{o}noma de M\'{e}xico (UNAM), México and\\ Centro de Investiagci\'{o}n en Matem\'{a}ticas (CIMAT), Guanajuato, M\'{e}xico.}
\thanks{$^{\ast\ast}$Supported by a MIUR SIR-grant ’Nonsmooth Differential Geometry’ (RBSI147UG4) and a DGAPA-UNAM Postdoctoral Fellowship of the Center of Mathematical Sciences.}

\author[M.~Zarei]{Masoumeh Zarei$^{\ast\ast\ast}$}
\address[M.~Zarei]{WW-Univsersität Münster,\\ 
Universit\"at Augsburg, Institut f\"ur Mathematik, 86135 Augsburg, Germany, and\\
Beijing International Center for Mathematical Research, Beijing, China.}
\thanks{$^{\ast\ast\ast}$Supported by DFG project AM 342/4-1 and DFG grant GA 2050/2-1 within the  Priority Program SPP-2026 ``Geometry at Infinity''.}

\thanks{The authors would like to thank the \textit{IV Encuentro Conjunto RSME-SMM} where initial talks leading to this work took place. In particular, we thank F.~Galaz-Garc\'ia and L.~Guijarro, organizers of the Riemannian/Metric geometry  special session for very useful conversations. The authors also thank the anonymous referee for her or his helpful comments.}


\subjclass[2010]{53C23, 57S15, 57S25}
\keywords{$4$-manifold, torus action, Alexandrov space}

\begin{abstract}
We obtain an equivariant classification for orientable, closed, four-dimensional  Alexandrov spaces admitting an isometric torus action. This generalizes the equivariant classification of Orlik and Raymond of closed four-dimensional manifolds with torus actions.  Moreover, we show that such Alexandrov spaces are equivariantly homeomorphic to  $4$-dimensional Riemannian orbifolds with isometric $\soso$-actions. We also obtain a partial homeomorphism classification.  
\end{abstract}

\maketitle
\setcounter{tocdepth}{1}



\section{Introduction}\label{S: Introduction}
In in the presence of a continuous action by a fixed Lie group $G$ on a  space $X$ 
there are usually two kinds of classifications.   By \emph{homeomorphism classification}, we mean characterizing all homeomorphism types of the spaces which admit such $G$-actions.  
 By \emph{equivariant classification}, we mean determining all possible actions of $G$ in terms of ``enough data''. This means being able to decide when two sets of data of the actions on two spaces  give rise to an equivariant homeomorphism between them. 

The study of $\soso$-actions on closed (i.e. compact and without boundary) $4$-manifolds was initiated by the work of Orlik and Raymond  in \cite{OrlRayI}, where they obtained  an equivariant classification for such actions on orientable manifolds under the added condition that the action  does not have  exceptional orbits. In particular, they showed that the smooth classification is equivalent to the topological classification, i.e. a continuous effective action on a closed topological $4$-manifold is equivariantly homeomorphic to a smooth effective action on a closed smooth $4$-manifold.  
 In a follow-up paper \cite{OrlRayII}, Orlik and Raymond extended the equivariant classification for the case when the action does admit exceptional orbits.  The classification of the homeomorphism type of smooth closed oriented $4$-manifolds with an effective $\soso$-action was carried out in \cites{OrlRayI,OrlRayII,Pao1,Pao2}. We provide this classification in Section~\ref{SS: Torus actions on $4$-manifolds}. 
For the non-orientable case, both the  equiviariant classifcation problem and the homeomorphism type classification problem  were addressed by Kim in \cite{Kim}. 
 
Alexandrov spaces (with curvature bounded form below) are complete length spaces with a lower curvature bound in the sense of triangle comparison. By Toponogov's Theorem, they generalize complete Riemannian manifolds with an uniform lower curvature bound.  Riemannian orbifolds (with a lower sectional curvature bound), orbit spaces of isometric  actions of compact Lie groups on Riemannian manifolds with sectional curvature bounded from below, and Gromov--Hausdorff limits of sequences of $n$-dimensional Riemannian manifolds with a uniform lower bound on their sectional curvature are examples of Alexandrov spaces.

Two-dimensional Alexandrov spaces are topological two-manifolds, possibly with boundary (see \cite{BurBurIva}*{Corollary 10.10.3}). Closed  three-dimensional Alexandrov spaces are either topological three-manifolds, or are homeomorphic to quotients of smooth three-manifolds by orientation reversing involutions with isolated fixed points (see \cite{GalGui} and cf. \cite{HaSe}). In higher dimensions though, similar general results do not exist. Considering Alexandrov spaces with a  large amount of \emph{symmetry} provides a systematic way of studying them.  This framework has been successfully used in the smooth category to study, for instance, Riemannian manifolds with  prescribed lower curvature conditions, such as positive Ricci or positive sectional curvature (see for example \cites{GrSea,GVZ,GWZ,GZ00,GZ02,Wil}). In particular,  many authors studied such spaces in the presence of a circle or a torus action either in low dimensions or in general    (see for example \cites{GaKe,GalSea11,GrSea94,GW,HarSea19,HsKl,Kle,SeaY}).

Since the group of isometries of a compact Alexandrov space is a compact Lie group \cite{FukYam}, one may assume the existence of an isometric compact Lie group action and consider  such spaces in analogy with the Riemannian case.
This allows us to bring  the theory of compact transformation groups into the context of Alexandrov spaces (see for example \cites{AZ, GalGui13, GalNun, GalSea, GalZar2,HaSe, HarSea, Nun, Yeroshkin}).

In particular, there exists a (non-unique) maximal dimension torus acting by isometries on $X$. When the space has curvature bounded below by $1$,  Harvey and Searle showed in \cite{HaSe} that the dimension of any maximal torus contained in the isometry group is bounded above by $\lfloor (n+1)/2\rfloor$. In particular, for a compact $4$-dimensional Alexandrov space with curvature bounded below by $1$, if we have  an effective $T^k$-action by isometries, then $k\leqslant 2$. 

In the present article we study the equivariant classification of effective, isometric $\soso$-actions on closed, orientable Alexandrov $4$-spaces.  We show that the techniques used in the study of smooth effective $\soso$-actions on smooth $4$-manifolds extend to this metric setting. Furthermore, this family  of  Alexandrov spaces contains a subfamily of spaces having maximal Abelian symmetry (see \cite{HaSe}).

Our main results are the following: Given a closed, orientable Alexandrov $4$-space $X$ with an effective, isometric $\soso$-action, first we find a set of invariants classifying $X$ up to equivariant homeomorphism. More precisely, we show that given  another closed orientable Alexandrov $4$-space $Y$ with an effective, isometric $\soso$-action, if $Y$ has the same set of invariants as $X$, then $Y$ is equivariantly homeomorphic to $X$. 
\begin{theorem}
\label{thm:invariants_introduction}
Let $X$ be a closed, orientable Alexandrov $4$-space admitting an effective, isometric $\soso$-action. Then the set of inequivalent (up to equivariant homeomorphism) effective, isometric $\soso$-actions on $X$ is in one-to-one correspondence with the set of unordered tuples of isotropic and topological invariants
\begin{linenomath}
\begin{equation*}
\big\{(b_1,b_2);\varepsilon;g; \big\{ \left\langle p_i,q_i\right\rangle \big\}_{i=1}^s; \big\{ \big( (a_{\ell},b_{\ell}),f_{\ell}\big) \big\}_{\ell=1}^{t} ; \big\{ (\alpha_j;\gamma_{j,1},\gamma_{j,2}) \big\}_{j=1}^k \big\}.
\end{equation*}
\end{linenomath}

\end{theorem}

We give the precise definitions and value constraints of the invariants stated in Theorem~\ref{thm:invariants_introduction} in Section \ref{S_Equivariant classification}. 

Recall that a \textit{topologically regular} Alexandrov space is an Alexandrov space with every \emph{tangent cone} homeomorphic to an Euclidean space.
In particular we recover the following result of Galaz-Garc\'ia \cite{Gal} for any topologically regular Alexandrov $4$-space, when comparing with the equivariant classification of effective $\soso$-actions on orientable $4$-manifolds.

\begin{corollary}
Let $X$ be a closed, orientable and topologically regular Alexandrov $4$-space with an effective, isometric $\soso$-action. Then $X$ is equivariantly homeomorphic to a smooth $4$-manifold and the $\soso$-action is equivalent to a smooth action.
\end{corollary}

Second, we show that the set of invariants in Theorem \ref{thm:invariants_introduction} is complete, that is, for each ``legal'' set of invariants we show that there exists a closed, orientable Alexandrov $4$-space $X$ with an effective, isometric $\soso$-action, such that it has the given set of invariants. In fact, we show that it is possible to realize such an Alexandrov $4$-space as a closed Riemannian $4$-orbifold, thus obtaining the following result. 

\begin{theorem}\label{thm: Alex 4 space with T^2 action is Orbifold, introduction}
Let $X$ be a closed, orientable Alexandrov $4$-space admitting an effective, isometric $T^2$-action. Then $X$ is equivariantly homeomorphic to the underlying topological space of a closed $4$-orbifold.
\end{theorem}

Theorem~\ref{thm: Alex 4 space with T^2 action is Orbifold, introduction} strengthens the results of \cite[Theorem~D]{HaSe} in dimension $4$, where the authors show that  positively curved Alexandrov $4$-spaces with $\soso$-action are equivariantly homeomorphic to $4$-dimenional $\soso$-orbifolds. In fact, our theorem shows that this is always the case, regardless of who is the lower curvature bound on the Alexandrov space.


Due to Theorem ~\ref{thm: Alex 4 space with T^2 action is Orbifold, introduction}, the homeomorphism classification for $4$-dimensional Alexandrov spaces with an effective isometric $\soso$-action is contained in the homeomorphism classification of orbifolds with an effective $\soso$-action. Such orbifolds were studied by Haefliger and Salem in \cite{HaSa}, but  they did not consider the homeomorphism classification. We recall that a $2k$-orbifold with an effective $T^k$-action is called a torus orbifold. These orbifolds were studied in \cites{GaKeRaWe,Wie}, for the case when the orbifold is simply-connected and rationally elliptic. A simply-connected topological space $X$ is called rationally elliptic if it satisfies $\dime_\QQ(H^*(X,\QQ))<\infty$ and $\dime_\QQ(\pi_\ast(X)\otimes\QQ)<\infty$. 

Regarding the homeomorphism type of closed, orientable Alexandrov $4$-spaces with effective and isometric $\soso$-actions we provide a basic topological recognition result. More precisely we show that any closed, orientable Alexandrov $4$-space with an effective, isometric $\soso$-action can be decomposed as an ``equivariant connected sum'' along subspaces homeomorphic to $S^2\times S^1$, where one of the pieces is an Orlik-Raymond $4$-manifold and the remaining pieces are ``simple'' Alexandrov $4$-spaces. We refer the reader to Section \ref{S:BASIC_TOPOLOGICAL_RECOGNITION} for the details. From this basic decomposition it follows that the topological classification could be answered, by identifying the homeomorphism type of these ``simple" Alexandrov pieces. The identification of such ``simple" pieces is beyond the scope of the present work. 

\begin{rem}
The proof of Theorem~\ref{thm:invariants_introduction} is mainly topological. The synthetic curvature condition is used implicitly to describe small balls around points (see Proposition~\ref{P: Normal space of directions}), and to study the local isotropy action (see Theorem~\ref{THM:SLICE}). Moreover, for Theorem~\ref{thm: Alex 4 space with T^2 action is Orbifold, introduction} we explicitly use the characterization of the spaces of directions.
\end{rem}

\vspace*{1em}
The organization of the article is as follows. In Section \ref{S:PRELIM} we recall the structural properties of closed Alexandrov $4$-spaces and the basic aspects and different notions of orientability of Alexandrov spaces. We also recall the main features of the theory of transformation groups on Alexandrov spaces. We conclude the section with a short summary of the classification of $\soso$-actions on $3$ and $4$-dimensional manifolds. In Section \ref{S:Orbit space and Orbit types} we describe the set of topologically singular points. In Section \ref{S:Local structure around topological regular points} we describe the local structure   of the orbit space of a closed, orientable Alexandrov $4$-space with an effective, isometric $\soso$-action, around the different orbit types. We complete the analysis of Section \ref{S:Local structure around topological regular points} by providing the topological structure and distribution of the different orbits in the orbit space. In Section \ref{S_Equivariant classification} we define the isotropy and topological invariants appearing in Theorem \ref{thm:invariants_introduction}, as well as the proof of this result. In Section \ref{S:Construction_Alexandrov_metric} we show the completeness of the invariants in Theorem \ref{thm:invariants_introduction} by constructing for every given set of invariants a closed, orientable Alexandrov $4$-space with a $\soso$-action. Finally in Section \ref{S:BASIC_TOPOLOGICAL_RECOGNITION} we define the equivariant gluings along $S^2\times S^1$ and prove the aforementioned topological recognition result.





\section{Preliminaries}
\label{S:PRELIM}
 
\subsection{Four-dimensional Alexandrov spaces}

We give a brief account of the main structural properties of closed four-dimensional Alexandrov spaces. For simplicity, throughout the article we refer to four-dimensional Alexandrov spaces as Alexandrov $4$-spaces. We refer the reader to \cites{ BurBurIva,BurGroPer} for a general introduction to the theory. 

Let $X$ be a closed (i.e. compact and without boundary) Alexandrov $4$-space. The space of directions $\Sigma_xX$ at each point $x\in X$ is a closed Alexandrov $3$-space with $\curv\geq 1$. It is a consequence of the classification of closed positively curved Alexandrov $3$-spaces (\cite[Theorem 1.1]{GalGui}) that the possible spaces of directions of a closed Alexandrov $4$-space are homeomorphic to $\mathbb{S}^3$, $\susp(\RR P^2)$ or a closed spherical $3$-manifold with non-trivial fundamental group (see \cite[Corollary 2.3]{GalGui}, see also \cite[Proposition~3.6 and Corollary~3.7]{HaSe}). Namely 

\begin{prop}[\protect{\cite[Corollary~2.3]{GalGui}}]\label{Prop:Space of dierections of 4Alex}
The space of directions of a $4$-dimensional Alexandrov space without boundary is homeomorphic to $\susp(\RR P^2)$ or to a spherical three-manifold. 
\end{prop}

A point whose space of directions is homeomorphic to $\mathbb{S}^3$ is called \textit{topologically regular}. Otherwise, the point is said to be \textit{topologically singular}. The set of topologically regular points of $X$ is an open and dense subset of $X$. 
The local structure of $X$ is determined by Perelman's Conical Neighborhood Theorem \cite{Per2} which states that every point $x\in X$ has a neighborhood pointed-homeomorphic to the cone over $\Sigma_xX$.

\subsection{Orientable Alexandrov spaces}
In this section we recall the notion of orientability for Alexandrov spaces (see \cite{Mitsuishi} for more details).

An Alexandrov space is called \textit{locally orientable} if  every point has an orientable neighborhood and \textit{locally non-orientable} otherwise. 
Equivalently, we say that $X$ is locally orientable if the space of directions at each point is orientable. That is  $H^n(X,X\smallsetminus\{p\}, \ZZ) \cong H^{n-1}(\Sigma_p X;\ZZ)\cong \ZZ$. Otherwise $X$ is locally non-orientable (\cite[Page~124]{Pet}).

Now we recall the global definition of orientability, following \cite{Mitsuishi}:
 
\begin{thm}[\protect{\cite[Theorem~1.8]{Mitsuishi},
}]\label{T_Orientability}
Let $X$ be a compact $n$-dimensional Alexandrov space. Then
 the following conditions are equivalent:
\begin{enumerate}[(a)]
\item\label{T__Orientability (a)}The manifold part $X_{top}$ of $X$ is orientable as an $n$-manifold;
\item\label{T__Orientability (b)}  $H_n(X; \Z)\cong \Z$;
\item\label{T__Orientability (c)}  $H^n(X; \Z)\cong \Z$ and the canonical morphism $H^n(X, X\smallsetminus\{x\}; \Z)\to H^n(X; \Z) $ is an isomorphism for any $x\in X$. 
\end{enumerate}
\end{thm}

\begin{definition}\cite[Definition~1.10]{Mitsuishi}
A compact Alexandrov space is called orientable if it satisfies one of the
conditions \eqref{T__Orientability (a)}--\eqref{T__Orientability (c)} listed in Theorem \ref{T_Orientability}.
\end{definition}

\begin{rem}
For a compact Alexandrov space of dimension at least $2$, the equivalence between items \eqref{T__Orientability (a)} and \eqref{T__Orientability (c)} in Theorem~\ref{T_Orientability}  it is also proven by Lemma~3.3 in \cite{HaSe}.
\end{rem}

\begin{rem}\label{R: remark orientability}
It follows directly from Proposition \ref{Prop:Space of dierections of 4Alex} that a closed Alexandrov $4$-space is locally orientable if and only if no space of directions is homeomorphic to $\susp(\mathbb{R}P^2)$. 
\end{rem}

\subsection{Group actions}
Let $X$ be a finite-dimensional Alexandrov space. Fukaya and Yamaguchi showed in \cite{FukYam} that, as in the Riemannian case, the isometry group of
$X$ is a Lie group. Moreover, if $X$ is compact then its isometry group is also compact. We consider isometric actions $G\times X \rightarrow X$ of a compact Lie group $G$ on $X$. The orbit of a point $x\in X$ is denoted by $G(x) \cong G/G_x$. Here, $G_x= \{ g\in G \ : \ gx= x	  \}$ is the isotropy subgroup of $x$ in $G$. The \textit{ineffective kernel} of the action is defined as the closed subgroup of $G$ given by $\cap_{x\in X} G_x$. If the ineffective kernel is trivial, we say that the action is \textit{effective}. In this article we assume all actions considered to be effective unless stated otherwise. Given a subset $A\subset X$ we denote its image under the canonical projection $\pi: X\rightarrow X/G$ by $A^*$. In particular we denote the orbit space by $X^*=X/G$.

The following generalization of the Principal Orbit Theorem for Alexandrov spaces was obtained in \cite[Theorem 2.2]{GalGui13}.

\begin{thm}[Principal Orbit Theorem]
\label{THM:Principal-orbit-theorem}
Let $G$ be a compact Lie group acting isometrically on an $n$-dimensional Alexandrov space $X$. Then there is a unique maximal orbit type and the orbits with maximal orbit type, the so-called principal orbits, form an open and dense subset of $X$.
\end{thm}

The following result stated in \cite[Proposition~4]{GalSea} gives a description of the tangent and normal spaces to the orbits of a Lie group action on an Alexandrov space.

\begin{prop}\label{P: Normal space of directions}
Let $X$ be an Alexandrov space admitting an isometric $G$-action and fix $x\in X$ with $\dim G/G_x > 0$. If $S_x \subset \Sigma_xX$ is the unit tangent
space to the orbit $G(x) = G/G_x$, and 
\[
S_x^{\perp} =\{w\in \Sigma_xX\mid \angle (v, w)=\pi/2 \, \,\text{for all} \,\,v \in S_x    \},
\]
then the following hold:
\begin{itemize}
\item[(1)] The set $S^{\perp}_x$ is a compact, totally geodesic Alexandrov subspace of $\Sigma_xX$ with curvature bounded below by $1$, and the space of directions $\Sigma_xX$ is isometric to the join $S_x\ast S^{\perp}_x$ with the standard join metric.
\item[(2)] Either $S^{\perp}_x$ is connected or it contains exactly two points at distance $\pi$.
\end{itemize}
\end{prop} 

We finally recall a crucial tool in the theory of transformation groups, the so-called Slice Theorem. This result provides a canonical identification of a small invariant neighborhoods of the orbit of through a given point $x$ with a fiber bundle over the orbit whose structure group is $G_x$. The fiber is the so-called slice, and for an Alexandrov space a slice can be identified with the cone over the space of unit normal directions to the orbit.

\begin{thm}[Slice Theorem, \protect{\cite[Theorem~B]{HaSe}} ] 
\label{THM:SLICE}
Consider a compact Lie group $G$ acting by isometries on an Alexandrov space X. Then for all $x\in X$, there is some $\varepsilon_0 > 0$ such that for all $\varepsilon < \varepsilon_0$ there is an equivariant homeomorphism 
\[
G\times_{G_x}K(S^{\perp}_x)\to B_{\varepsilon}(G(x)),
\]
where $K(S^{\perp}_x)$ is the cone over $S^{\perp}_x$.
\end{thm}

\subsection{Torus actions on $4$-manifolds}\label{SS: Torus actions on $4$-manifolds}
We conclude the  preliminaries section with a brief review of the equivariant and homeomorphism classifications due to Orlik, Raymond, Pao and Kim of closed $4$-manifolds admitting effective torus actions for the sake of completeness.  Observe that closed $4$-manifolds with any Riemannian metric are Alexandrov spaces, since each Riemannian metric has a global lower bound for its sectional curvature. Thus, by taking an equivariant Riemannian metric,  the classification of closed orientable $4$-manifolds with a torus actions, is a particular case of the main results of the present work.

Let us begin by recalling the equivariant classification of closed, orientable $4$-manifolds with an effective torus action.

\begin{thm}[\protect{\cites{OrlRayI,OrlRayII}}]
Let $M$ be a closed, orientable $4$-manifold admitting a $\soso$-action by homeomorphisms. Then the set of inequivalent (up to equivariant homeomorphism) effective $\soso$-actions on $M$ is in one-to-one correspondence with the set of unordered tuples of isotropy and topological invariants
\[
\left\{(b_1,b_2);\varepsilon; g; s;t; \big\{ \left\langle p_i,q_i\right\rangle \big\}_{i=1}^s; \big\{ \{a_{l},b_{l}\} \big\}_{l=1}^{t} ; \big\{ (\alpha_j;\gamma_{j,1},\gamma_{j,2}) \big\}_{j=1}^k   \right\}
\]
\end{thm}

To avoid repetition, the definition of the invariants appearing in the previous theorem, and their admissible values, are defined in Section \ref{S_Equivariant classification} for the more general case of isometric, effective torus actions on Alexandrov $4$-spaces. Let us only mention here that each invariant of the form  $\{a_{l},b_{l}\}$ corresponds to the invariant $\big( (a_{l},b_{l}), \pm 1\big)$ as formulated in Section \ref{S_Equivariant classification}. This is due to the fact that there are no topologically singular points in the manifold case.  

Although it is not needed for the proofs of the results stated in Section~\ref{S: Introduction}, for the sake of completeness, we recall here that Kim obtained the equivariant classification for closed non-orientable $4$-manifolds admitting an effective torus action  \cite{Kim}.  

Different orbit types correspond to the possible isotropy groups of the torus action.  An orbit is called a \textit{principal orbit} or \textit{$P$-orbit}  if the isotropy group of a point (and hence for any point) contained in the orbit is trivial. The set of points in $M$ whose orbits are $P$-orbits is denoted by $P$.  Theorem \ref{THM:Principal-orbit-theorem}, i.e the Principal Orbit Theorem, states that $P$ is an open and dense subset of $M$. Orbits with non-trivial finite isotropy acting by preserving the local orientation are called \textit{exceptional orbits} or \textit{$E$-orbits}. The subset of $M$ of points on $E$-orbits is called $E$. For the case of $4$-dimensional manifolds, The Slice Theorem~\ref{THM:SLICE} implies that  each $E$-orbit is isolated from other $E$-orbits, since we have a free linear action of a finite subgroup of $\soso$ on the circle $S^1$ (see \cite[p.~534]{OrlRayI} for the list of possible finite subgroups). The orbit of a point whose isotropy is isomorphic to the circle is called a \textit{circular orbit} or \textit{$C$-orbit}. We denote the set of points of $M$ contained in $C$-orbits by $C$. Finally, the set of points which are fixed points of the action (that is, having isotropy subgroup equal to $\soso$) are denoted by $F$. 

The orbit space $M^*$ is a $2$-manifold (possibly with boundary) in which the interior points are made up by $P^*\cup E^*$. Moreover, $E^*$ is a finite set. The boundary, which may be empty, is equal to $F^*\cup C^*$.



For the case in which $F\neq \emptyset$ some families of $4$-manifolds appear naturally as ``building blocks'' of $4$-manifolds admitting effective $\soso$ actions, and are presented next. Let $M$ be a closed $4$-manifold admitting an effective $\soso$-action. We say that $M$ is of \emph{type}: 
\begin{itemize}
\item[(i)] \emph{$R(m,n)$}, if $\gcd(m,n)=1$, and the orbit space $M^*$ is homeomorphic to an annulus where the inner circle boundary  is composed of orbits with isotropy 
\[
G(m,n) =\{ (\varphi, \theta)\ \mid\ m\varphi + n\theta =0, \ \mathrm{gcd}(m,n)=1\}.
\] 
The outer circle boundary consits of two arcs joined by their endpoints where one of them consists of orbits with $G(0,1)$ isotropy, while the other consits of orbits with $G(1,0)$ isotropy.

\item[(ii)] \emph{$T(m,n;m',n')$}, if $\gcd(m,n)=\gcd(m',n')= 1$, $mn-m'n'=\pm 1$, and the orbit space $M^*$ is homeomorphic to an annulus where the inner circle boundary  is composed of two arcs joined at their endpoints: one of them has $G(m,n)$ isotropy while the other has $G(m',n')$ isotropy.  The outer circle boundary consists of two arcs joined by their endpoints where the orbits in one of them has $G(0,1)$ isotropy, while the other arc consists of orbits with $G(1,0)$ isotropy.

\item[(iii)] \emph{$L(n;p,q;m)$}, if $\gcd(n,p,q)=1$, and the orbit space is homeomorphic to a $2$-disk with  the center point corresponding to an $E$-orbit with (non-normalized) Seifert invariants $(n;p,q;m)$ (see Section \ref{S_Equivariant classification}) and with the boundary consisting of two arcs joined by their corresponding endpoints. One arc consists of orbits with isotropy $G(m,n)$ while the other arc consists to orbits with isotropy $G(p,q)$. 
\end{itemize}

The  homeomorphism type of these manifolds was completely described by Pao, \cite{Pao1}, and they are the following (see \cite[Theorem~III.3]{Pao1}): 
\begin{itemize}
\item[(i)] For each $m,n$ with $\gcd(m,n)=1$, 
\[ 
 R(m,n) =
\begin{cases} 
 S^2\times S^2 \# S^3 \times S^1  & \text{if}\ mn\equiv 0 \ \ (\text{mod $2$}),\\
 \mathbb{C}P^2\# \overline{\mathbb{CP}^2} \# S^3\times S^1 & \text{if}\ mn\equiv 1 \ \ (\text{mod $2$}).
\end{cases}
\]
\item[(ii)] Observe that since $mn-m'n'=\pm 1$ then at least two of the four integers $m,n,m', n'$ must be odd and one is even. Take $t=0$ if there are two even integers, and $t=1$ otherwise. Then (see  \cite[Theorem~IV.1]{Pao1}), 
\[ 
 T(m,n;m',n') =
\begin{cases} 
 S^2\times S^2 \# S^2\times S^2 \# S^3 \times S^1  & \text{if}\ t=0,\\
 \mathbb{C}P^2\# \overline{\mathbb{CP}^2}\# \mathbb{C}P^2\# \overline{\mathbb{CP}^2}\# S^3\times S^1 & \text{if}\ t=1.
\end{cases}
\]
\item[(iii)] For every $m,n,p, q$ with  $\gcd(n,p,q)=1$, there are two homeomorphism types (see \cite[Theorem~V.1]{Pao1}):
\[ 
 L(n;p,q;m) =
\begin{cases} 
  L(n;0,1;1)  & \text{or}\\
  L(n;1,1;1). &  
\end{cases}
\]
We point out that when $n$ is odd, then $L(n;0,1;1)$ is homeomorphic to $L(n;1,1;1)$. But for the case when $n$ is even, the spaces $L(n;0,1;1)$ and $L(n;1,1;1)$ do not have the same homotopy type (see \cite[Theorem~V.2]{Pao1}). Both of these manifolds result from a $\soso$-equivariant surgery of type $(2,3)$ on $S^1\times S^3$. By performing a surgery  of type $(2,3)$ on one element in the fundamental group of the manifold, we can produce at most two different manifolds (see \cite[Corollary V.8]{Pao1}).  
\end{itemize}

With these definitions in hand, we now summarize the whole topological classification of closed, orientable $4$-manifolds with  an effective $\soso$-action.

\begin{thm}[\protect{\cites{OrlRayI,OrlRayII,Pao1,Pao2}}]\label{T: Homeomorpshism type of orientable closed 4-manifolds with effective T2 action}
Let $M$ be a closed, orientable $4$-manifold admitting a $\soso$-action by homeomorphisms. Then the following hold true:
\begin{enumerate}[(1)]
\item\label{T: Homeomorpshism type of orientable closed 4-manifolds with effective T2 action 1}  If $M$ is simply-connected, then $E=\emptyset$, $F\neq \emptyset$ and $M$ is equivariantly homeomorphic to an equivariant connected sum of a finite number of copies of $\mathbb{CP}^2$, $\overline{\mathbb{CP}^2}$, $\mathbb{CP}^2\# \overline{\mathbb{CP}^2}$, $S^2\times S^2$ and $S^4$. This connected-sum decomposition is not necessarily unique (see \cite[Section~5]{OrlRayI}). 
%

\item\label{T: Homeomorpshism type of orientable closed 4-manifolds with effective T2 action 2} Assume that $F\cup C = \emptyset$. Then $M$ admits a fibering over $S^1$ where the fiber is a $3$-dimensional Seifert manifold (see \cite[Section~2]{OrlRayII}).


\item\label{T: Homeomorpshism type of orientable closed 4-manifolds with effective T2 action 3} If $M$ is not simply-connected and $F\neq \emptyset$ then $M$ is equivariantly homeomorphic to an equivariant connected sum of the simply-connected manifolds appearing in item (1) and \textit{elementary} $4$-manifolds of types $L$, $R$ and $T$. The decomposition as a connected sum is possibly non-unique (see \cite[Section~3]{OrlRayII}). 

\item\label{T: Homeomorpshism type of orientable closed 4-manifolds with effective T2 action 4} Suppose $F=\emptyset$ and $C\neq \emptyset$, then the following hold (see \cite{Pao2}):
\begin{itemize}
\item[(a)] If there are $C$-orbits with isotropies $G(m_1,n_1)$ and $G(m_2,n_2)$ such that $m_1n_2-m_2n_1=\pm 1$, then $M$ is an equivariant connected sum of copies of $S^4$, $S^2\times S^2$, $\mathbb{C}P^2$, $\overline{\mathbb{C}P^2}$, $S^3\times S^1$ and manifolds of type $L$. 
\end{itemize}
In the remaining cases, that is, when there are no mutually orthogonal $C$-orbits, the topological classification is not yet complete. However, the following properties hold:
\begin{itemize}
\item[(b)] If $M$  and the action have the following set of  invariants 
\[
\left\{o,g,s, -, \left\{ \left\langle p_i,q_i\right\rangle\right\}_{i=1}^s ; \big\{ (\alpha_j;\gamma_{j,1},\gamma_{j,2}) \big\}_{j=1}^k \right\},
\] 
then, for the second Steiffel-Whitney class $\omega_2(M)$ of $M$, we have $\omega_2(M)\neq 0$ if and only if there exist integers $i$ and $j$ with $2\leq i$ and $j\leq s$ such that 
\begin{align*}
p_i\equiv p_j\equiv q_i\equiv 1 & \ \ (\text{mod $2$}),\\
q_j\equiv 0 & \ \ (\text{mod $2$}). 
\end{align*}
\item[(c)] The integers $2g+s$, $k$, $\alpha_1, \ldots \alpha_k$ and $m=\gcd(p_1,\ldots,p_s)$ determine the fundamental group of $M$ (see \cite[Section 3]{Pao2}). 

\end{itemize}
\end{enumerate}
\end{thm}

It is worth remarking that the case $F\cup C = \emptyset$ was treated in \cite{OrlRayII} as a particular case of \cite{ConRay}, where the authors classify holomorphic Seifert fiberings through so called \textit{Bieberbach classes} (see also the work of Ci\v{s}ang in \cite{Zie}). We refer the reader to \cite{ConRay} and \cite[Section 2]{OrlRayII} for a detailed exposition. 
%

In view of the non-uniqueness of the connected sum decompositions in items \eqref{T: Homeomorpshism type of orientable closed 4-manifolds with effective T2 action 1} and \eqref{T: Homeomorpshism type of orientable closed 4-manifolds with effective T2 action 3}  in Theorem~\ref{T: Homeomorpshism type of orientable closed 4-manifolds with effective T2 action} above, Pao defined the so-called \textit{normal decomposition} for $4$-manifolds which admit an effective $\soso$-action with fixed points, and proved that this decomposition is unique. We say that a connected sum decomposition of such a $4$-manifold $M$ into summands $S^4$, $S^2\times S^2$, $S^3\times S^1$, $\mathbb{C}P^2$, $\overline{\mathbb{C}P^2}$, and manifolds of type $L$, is a normal decomposition if the number of copies of $S^4$, $\mathbb{C}P^2$, $\overline{\mathbb{C}P^2}$ and the manifolds of type $L$, denoted by $N(S^4)$, $N(\mathbb{C}P^2)$, $N(\overline{\mathbb{C}P^2})$, $N(L(n;0,1;1))$, and $N(L(n;1,1;1))$ respectively, satisfy the following conditions: 
\begin{itemize}
\item[(i)] $N(S^4)=1$,
\item[(ii)] $N(\overline{\mathbb{C}P^2})=0$ or $N(\overline{\mathbb{C}P^2})= N(\mathbb{C}P^2)=1$, 
\item[(iii)] if $N(\mathbb{C}P^2)\neq 0$ then $N(L(n;1,1;1))=0$ for all $n=2,3,\ldots,$
\item[(iv)] $N(L(n;1,1;1))=0$ for all $n=3,5,7,\ldots$.
\end{itemize}  

\begin{thm}[\protect{\cite[Theorem II.4.2]{Pao1}}]
Let $M$ be a closed, orientable $4$-manifold admitting a $\soso$-action by homeomorphisms with $F\neq \emptyset$. Then $M$ has a unique normal decomposition.  
\end{thm}

\subsection{Cohomogeneity-one $3$-manifolds}\label{SS:Cohomogeneity-one 3-manifolds}

Throughout this article the classification of coho\-mo\-geneity-one closed $3$-manifolds due to Mostert \cite{Mos} and  Neumann \cite{Neu} is used and  we review it here as well for the convenience of the reader. 

Let $G$ be a compact, connected Lie group acting by diffeomorphisms on a closed smooth $3$-manifold $M$ with cohomogeneity one. Then the orbit space $M^\ast$ is homeomorphic to either a circle $S^1$, or to a closed bounded interval $I=[-1,1]$. In the first case, there is a single orbit type with corresponding isotropy $H$. Thus, $M$ is a fiber bundle over $S^1$ with fiber $G/H$ and structure group $N(H)/H$. In this case, $M$ can be classified by the components of $N(H)/H$. 

In the case that the orbit space $M^\ast$ is homeomorphic to $I$, the two endpoints of $I$ correspond to the only two non-principal orbit types (which could be the same), with isotropy subgroups $K^-$, $K^+$, while the interior corresponds to  principal orbit types with isotropy $H$. Moreover, there is a group diagram made up of inclusions 
\begin{equation*}
		\xymatrix{ & G    & \\
			K^{-} \ar[ru]^{j_{-}} & & K^{+} \ar[lu]_{j_{+}}\\
			& H  \ar[lu]^{i_{-}} \ar[ru]_{i_{+}}& }
\end{equation*}	
where $K^{\pm}/H$ are isometric to positively curved $r_i$-spheres, with $0\leqslant r_i \leqslant 2$. Using the Slice Theorem, it is possible to prove that $M$ is equivariantly diffeomorphic to the union of two fiber bundles $ G\times_{K^+}D^{r_1+1}$ and  $ G\times_{K^+}D^{r_2+1}$, where $D^{r_i+1}$ is the closed unit disk in $\mathbb{R}^{r_i+1}$. Here the union is made via an equivariant diffeomorphism between the  boundaries of $G\times_{K^{\pm}}D^{r_i+1}$, which are diffeomorphic to a principal orbit $G/H$. Therefore, $M$ can be classified via this construction by the components of the double quotient $N_1 \backslash N(H) / N_2 $, where $N_1= N(H)\cap N(K^{-})$ and $N_2= N(H)\cap N(K^{+})$. 
 
Using these structure results, Mostert and Neumann obtained that either $G=\soso$ with $H=\{e\}$, or $G=\mathrm{SO}(3)$ with $H=\mathrm{O}(2)$, or $H= \mathrm{SO}(2)$. In Table~\ref{TBL:cohom1-3mflds} we collect the possibilities for the homeomorphism type of $M$, when $M^\ast$ is a closed interval, together with the corresponding orbit structures. Here, by orbit structure we mean the ordered tuple $(H,K^-,K^+)$. We denote the Klein bottle as  $\mathrm{Kl}$, the non-orientable $S^2$ bundle over $S^1$ by $S^2\tilde{\times} S^1$,  a lens space as $L(p,q)$,  the closed M\"obius band as $\mathrm{Mb}$, and $A$ denotes the manifold obtained by gluing $\mathrm{Mb}\times S^1$ to $S^1\times \mathrm{Mb}$ via the identity map along the boundary $S^1\times S^1$.

Galaz-Garcia and Zarei proved that for a topological manifold of dimension at most $4$, any continuous group action by a Lie group is equivariantly diffeomorphic to a smooth action on a smooth manifold.

\begin{thm}[\protect{\cite[Corollary~E]{GalZar}}]\label{T: For dim at most 4, every topol coho 1 action is smooth}
Consider a topological manifold $M$ with an (effective) topological action of a compact Lie group $G$ of cohomogeneity one. If $M$ has dimension at most $4$, then $M$ is equivairiantly homeomorphic to a smooth manifold $N$ with an effective smooth action by $G$.
\end{thm}

Therefore for dimension $3$, Table~\ref{TBL:cohom1-3mflds} gives a complete list of continuous actions of $\soso$ on topological $3$-manifolds.


\begin{table}[ht]
\resizebox{\textwidth}{!}{%
\begin{tabular}{|c|c|c|}
\hline
\textbf{Homeomorphism type of $M$} & \textbf{Orbit Structure for $G=\soso$} & \textbf{Orbit Structure for $G=\mathrm{SO}(3)$} \\ \hline
 $T^3$ & $\{e\}$ & $\emptyset$  \\ \hline
 $\mathrm{Kl}\times S^1$ & $\left(\{e\}, \mathbb{Z}_2\times \{1\}, \mathbb{Z}_2\times \{1\}  \right)$ & $\emptyset$ \\ \hline
 $A$ &  $\left(\{e\}, \mathbb{Z}_2\times \{1\}, \{1\}\times\mathbb{Z}_2   \right)$ & $\emptyset$ \\ \hline
$\mathbb{R}P^2\times S^1$ & $\left(\{e\}, \mathbb{Z}_2\times \{1\}, G(1,0)   \right)$ & $\left( \mathrm{O}(2), \mathrm{O}(2), \mathrm{O}(2)  \right)$ \\ \hline 
$S^2\widetilde{\times}S^1$ & $\left(\{e\}, \mathbb{Z}_2\times \{1\}, G(1,0)   \right)$ & $\left( \mathrm{SO}(2), \mathrm{SO}(2), \mathrm{SO}(2)  \right)$ \\ \hline
$S^2\times S^1$ & $\left(\{e\}, G(1,0), G(1,0) \right)$ & $\left( \mathrm{SO}(2), \mathrm{SO}(2), \mathrm{SO}(2)  \right)$\\ \hline
$L(p,q)$ & $\left(\{e\}, G(1,0), G(p,q) \right)$ with $\gcd(p,q)=1$ & $\emptyset$ \\ \hline
$S^3$ & $\left(\{e\}, G(1,0), G(0,1) \right)$ & $\left(\mathrm{SO}(2), \mathrm{SO}(3), \mathrm{SO}(3)  \right)$\\ \hline
$\mathbb{R}P^3$ & $\left(\{e\}, G(1,0), G(1,2) \right)$ & $\left( \mathrm{SO}(2), \mathrm{O}(2), \mathrm{SO}(3) \right)$ \\ \hline
$\mathbb{R}P^3\# \mathbb{R}P^3$ & $\emptyset$  & $\left(\mathrm{SO}(2), \mathrm{O}(2), \mathrm{O}(2) \right)$ \\ \hline
\end{tabular}%
}
\caption{Cohomogeneity-one $3$-manifolds}
\label{TBL:cohom1-3mflds}
\end{table}


\section{Local structure of topological singular points}
\label{S:Orbit space and Orbit types}



In general, the set of topologically singular points $\Sing_X$ of an Alexandrov space $X$ may be wildly arranged in $X$. However, in the presence of an isometric $\soso$-action the set $\Sing_X$ acquires more structure. We begin by pointing out that since $\soso$ acts by isometries, then $\Sing_X$ is an invariant subset of $X$ under the action of $\soso$. 

Observe that $\Sing_X$ is invariant under the isometric action of $T^2$. We split the set $\Sing_X$ into the disjoint union of the following subsets:

\begin{eqnarray*}
\Sing^0_X & = & \left\{x\in \Sing_X \mid \dim G(x)=0 \right\}\\
\Sing^1_X & = & \left\{x\in \Sing_X \mid \dim G(x)=1 \right\}\\
\Sing^2_X & = & \left\{x\in \Sing_X \mid \dim G(x)=2 \right\}
\end{eqnarray*}

%

\subsection{Orbits of points in $\Sing_X$}
\label{subsec-structure-S_X}

As mentioned before, the canonical projection $\pi:X \to X^\ast$ is a submetry. Therefore a small neighborhood $B_{\varepsilon}(x^{\ast})\subset X^\ast$ centered at $x^\ast\in X^\ast$ is homeomorphic to $\pi(B_{\varepsilon}(x)) = B_{\varepsilon}(x)^\ast$, where $\pi(x)=x^\ast$.  The Conical Neighborhood Theorem of Perelman (see \cite[Theorem 0.1]{Per}, \cite[Theorem 6.8]{Kap}) together with Proposition~\ref{Prop:Space of dierections of 4Alex}
imply that  $B_{\varepsilon}(x)$ is either homeomorphic to a cone over the suspension of the real projective plane $K(\susp(\mathbb{R}P^2))$, or a cone over a spherical manifold.  Furthermore, Theorem \ref{THM:SLICE} yields that $B_{\varepsilon}(x)^\ast$ is homeomorphic to $K(S^{\perp}_x/G_x)$. 

To further investigate the local structure of $X^\ast$ at orbits of points in $\Sing_X$ we split the analysis using the dimension of the orbits.

\begin{prop}
\label{P: S_X is finite collection of circle and points} Let $X$ be a closed, orientable Alexandrov $4$-space with an effective, isometric $\soso$-action. Then $\Sing_X=\Sing^0_X$, that is, the set of topologically singular points of $X$ consists 
only of fixed points.  
\end{prop}

\begin{proof}
Assume that there exists $x\in \Sing_{X}^{2}$. By Proposition~\ref{P: Normal space of directions} we have $\Sigma_xX = S_x \star S_x^\perp$. Since  $\dim G(x) = 2$, $S_x = S^1$, and thus the space of normal directions $S_x^\perp$ is a positively curved Alexandrov $1$-space. Thus $S_x^\perp$ is  homeomorphic to $S^1$. But this implies that the space of directions $\Sigma_x X = S^1\star S^1 = S^3$. Thus $x$ is a regular point, which is a contradiction.

Assume there exists $x\in \Sing_{X}^{1}$. By Proposition~\ref{P: Normal space of directions} we have $\Sigma_xX = S_x \star S_x^\perp$. Since  $\dim G(x) = 1$, $S_x$ consists of two points and hence  the space of normal directions $S_x^\perp$ is a positively curved Alexandrov $2$-space. Thus $S_x^\perp$ is either homeomorphic to $S^2$ or $\RP^2$. Therefore,  we have that $\Sigma_xX$ is either homeomorphic to $\susp (S^2) = S^3$ or $\susp( \RP^2)$. The first case means that $x$ is a topologically regular point, which is a contradiction. Moreover,  the case where $\Sigma_xX = \susp (\RP^2)$ contradicts the orientability of $X$. Hence, we conclude that $ \Sing_{X}^{1}=\emptyset$. Recall that the isotropy subgroups, are closed subgroups of $\soso$. Since for $x\in \Sing_{X}^0$ we have an homeomorphism between $\soso(x)$ and $\soso/\soso_x$, we conclude that $\soso_x$ is a two-dimensional subgroup. This implies that the isotropy subgroup is $\soso$, and $x$ is a fixed point.
\end{proof}

In the following lemma we describe the normal space of directions for each orbit type of the effective action by isometries of $\soso$ on $X$. 

\begin{lem}\label{L: Normal spaces of directions}
 Let $X$ be a closed, orientable Alexandrov $4$-space admitting an effective isometric $T^2$-action. Then the normal space of directions to an orbit must be one of the following:
\begin{enumerate}[(a)]
\item\label{L: Normal spaces of directions-principal} For a two dimensional orbit, 
the normal space of directions is homeomorphic to $S^1$;
\item\label{L: Normal spaces of directions-circle} For a  one dimensional orbit, 
the normal space of directions is homeomorphic to $S^2$; and
\item\label{L: Normal spaces of directions-fixed point} For a zero dimensional orbit, the space of directions is homeomorphic to a $3$-dimensional spherical manifold.
\end{enumerate}
\end{lem}

\begin{proof}
We begin by proving \eqref{L: Normal spaces of directions-principal}. Let $x\in X$ be an element contained in a two dimensional orbit. Thus its normal space of directions is  a closed positively curved  $1$-dimensional Alexandrov space by Proposition~\ref{P: Normal space of directions}; i.e.\ it is homeomorphic to the circle $S^1$. Consider now $x\in X$ contained in a one dimensional orbit. Observe that by Proposition~\ref{P: S_X is finite collection of circle and points}, it is  topologically regular. Moreover, the tangent sphere to the orbit consists of two points. Thus by Proposition~\ref{P: Normal space of directions}, its normal space of directions is a closed $2$-dimensional Alexandrov space of curvature bounded below by $1$. Hence $S^{\perp}_{x}$ is $S^2$.


We consider $x\in X$ in a zero dimensional orbit. Then, $x$ is a fixed point and Proposition~\ref{Prop:Space of dierections of 4Alex} gives that $\Sigma_x X$ is either homeomorphic to a spherical $3$-manifold, or to $\susp(\RP^2)$. The orientability of $X$ then implies that $\Sigma_xX$ is a spherical $3$-manifold.

\end{proof}

\section{Local structure of the orbit space}
\label{S:Local structure around topological regular points}

In this section, we give the local characterization of the action  by describing the structure the  orientable $4$-dimensional Alexandrov space$X^\ast$ around the orbits of an isometric effective action of $\soso$ . We stress that most of the contents of this section are known (see for example, \cite{OrlRayI, OrlRayII}), but for the sake of completeness we review them here. From Lemma~\ref{L: Normal spaces of directions}, it follows that  an orientable $4$-dimensional Alexandrov  space $X$ with an isometric effective $\soso$-action  is in fact a rational cohomology manifold. Note that $X\setminus \Sing^0_X$ is an orientable manifold by Theorem~\ref{T_Orientability}. In particular, $X\setminus \Sing_X^0$ is an integral cohomology  manifold.   Such spaces were treated  in \cite[Appendix]{OrlRayI}. However, thanks to the Slice Theorem (Theorem \ref{THM:SLICE}), we have a clear picture of the local structure of an orientable  $4$-dimensional Alexandrov space with an effective isometric $\soso$-action and its orbit space. This enables us  to give a more precise description of the action.


The possible isotropy groups correspond to closed subgroups of $\soso$. In order to set notation, let us assume that $\soso$ is parametrized by $(\varphi,\theta)$ with $0\leq\varphi\leq 2\pi$ and $0\leq \theta\leq 2\pi$. Then the closed subgroups of $\soso$ consist of $\{1\}$, $\ZZ_n$, $\ZZ_n\times \ZZ_m$, the subgroups of the form
\[
G(m,n)=\{ (\varphi, \theta)\ \mid\ m\varphi + n\theta =0, \ \mathrm{gcd}(m,n)=1\},
\]
$G(m,n)\times \ZZ_p$ and $\soso$ itself. Note that $G(m,n)$ is isomorphic to the circle group $\mathrm{SO}(2)$. 

We write, $G=\soso$, and assume it acts isometrically and effectively on a closed, oriented Alexandrov $4$-space $X$. We list all the possible isotropy groups, slices and the local picture of the orbit space  . 


\begin{prop}\label{P: isotropy groups}
Let $X$ be an orientable closed Alexandrov $4$-space with an effective isometric $T^2$-action. Then the possible isotropy groups are either  the trivial one, $T^2$,  or of the form $\ZZ_p$,  or $G(m,n)$.
\end{prop}

\begin{proof}
We only need to show that $G(m,n)\times \ZZ_p$ and $\ZZ_n\times\ZZ_m$ cannot be the isotropy group of any $x\in X$. We consider two cases: When $x$ is not a topologically singular point, then the analysis done by Orlik and Raymond in \cite{OrlRayI,OrlRayII} applies. That is, for a topological regular point $x\in X$ the possible isotropy groups are the trivial subgroup, $T^2$, or one of the form $\ZZ_p$, or $G(m,n)$ with $\gcd (m,n) = 1$. From  Proposition~\ref{P: S_X is finite collection of circle and points}, we know that for the case when $x$ is a topological singular point, it is a fixed point which concludes the proof.
\end{proof}

\subsection{Local structure of $X^{\ast}$ around points in $\Sing_{X}$}
\label{Sub: Local structure of X^* around points in S_X^0}


Given a point $x\in \Sing_X$ we describe a small neighborhood in $X^\ast$ of its orbit $x^\ast\in X^\ast$, and the isotropies of the orbits contained in this neighborhood. First, recall that $x$ is a fixed point of the $\soso$-action on $X$, and $\Sigma_x X$ is a spherical space form by Lemma~\ref{L: Normal spaces of directions}, Part \eqref{L: Normal spaces of directions-fixed point}. Moreover, it is not $ S^3$ as $x$ is a singular point. Since $\soso$ acts on $\Sigma_x X$ isometrically and effectively, the action is a continuous action of cohomogeneity one on a spherical space form. 

Moreover, since the space of directions $\Sigma_x X$ is a  positively curved Alexandrov space, by the generalized Bonnet--Myers Theorem for positively curved Alexandrov spaces \cite[Theorem~2.10]{HaSe}, we conclude that $\Sigma_x X$ has finite fundamental group. Since the fundamental group surjects onto the fundametal group of $\Sigma_x X/\soso$, the orbit space $\Sigma_xX /\soso$ is a closed interval. 

By the classification of cohomogeneity one actions on $3$-dimensional manifolds \cite{Neu,GalZar} and by the fact that $\Sigma_x X$ is a spherical space form it follows that 
 $\Sigma_xX$ is a lens space $L(q,p)$, and the orbit space $\Sigma_xX/\soso$ has group diagram of the form $(\soso, G(1,0),G(p,q),\{1\})$, for some appropriate splitting of $\soso$ (see Table~\ref{TBL:cohom1-3mflds}).  Furthermore, it follows from \cite{Neu} that  $(p,q)\neq (1,0)$ and $(p,q)\neq (0,1)$. Otherwise we would have that $\Sigma_xX$ is a 3-sphere or $S^2\times S^1$, respectively, which yields a contradiction, since $x\in\Sing_X$ and $\Sigma_x X$ has finite fundamental group. Therefore, for a general  group diagram  of the form $(\soso, G(a,b), G(c,d), \{1\})$ this implies that 
 we have $ad-bc \neq 0$ and $ad-bc \neq \pm 1$. Figure \ref{fig:orbit-space-t2-action-on-lens-space} depicts the local structure of $X^\ast$ around $x\in \Sing^0_X$.

\begin{figure}[ht]
\begin{center}
\begin{tikzpicture}[rotate=90,scale = 0.6]
\fill[color=gray!15!white] (-1.99,3.98) -- (0,0) -- (1.99,3.98)-- cycle;
\draw [] (-1.99,3.98) -- (0,0) -- (1.99,3.98)-- cycle;
\draw[color = black, fill=black] (0,0) circle (2pt) node () [right]{$ \ \soso$};
\node at (0,1.1)  {$\ \ \ \ x^{*}$};
\node at (1,1.8) [above] {$\ \ \ \  G(a,b)$};
\node at (-1,1.8) [below] {$\ \ \ \ G(c,d)$};
\node at (0,4) [left] {$\{1\}$};
\end{tikzpicture}
\caption{Neighborhood in $X^\ast$ around the orbit of a topologically singular point $x\in \Sing_{X}^{0}$.}
\label{fig:orbit-space-t2-action-on-lens-space}
\end{center}
\end{figure}

Moreover, since $X$ is compact, we conclude that the set $\Sing_X$ consists of a finite set of fixed points.

\begin{lem}
Let $X$ be a closed, orientable Alexandrov $4$-space with an effective, isometric $\soso$-action. Then $\Sing_X$ consist of a finite number  of fixed points.  
\end{lem}

\subsection{Fixed points}\label{Sub: Regular Fixed points}

 Assume that $x\in X\smallsetminus \Sing_X$ is a point with isotropy group $G_x=\soso$, i.e. $x$ is a fixed point. To describe the local picture of the orbit space around $x^\ast$, we need to examine the effective action of $\soso$ on the normal space of directions to the orbit, that is $S_x^\perp= S^3$. A priori the action of $\soso$  on $S^3$ is only continuous, but since it is of cohomogeneity one, by Theorem~\ref{T: For dim at most 4, every topol coho 1 action is smooth}, 
there is an equivariant homeomorphism to the standard smooth $\soso$-action on $S^3$. We recall that the standard action of  $\soso$ on $S^3$ is as follows:
 
\begin{align*}
\soso\times S^3&\longrightarrow S^3\\
(\varphi, \theta), (z_1, z_2))&\longmapsto (e^{i m \varphi} e^{i n \theta}z_1, e^{i p \varphi} e^{i q \theta}z_2),
\end{align*}
where we regard $S^3\subseteq \CC^2$, and assume $\mathrm{gcd}(m, n)=1$, and $\mathrm{gcd}(p, q)=1$. Note that this action is effective if and only if $mq-np=\pm 1$.
The group diagram associated to the cohomogeneity one action of $\soso$ on $S^3$ is $(\soso, G(m, n), G(p, q), \{1\})$. Then we have the local picture of $X^\ast$ around $x^\ast$ as in Figure~\ref{FIG_Regular_Fixed}.

%

\subsection{One dimensional orbits}
\label{SS:One dimensional orbits}

Assume that $x\in X\smallsetminus \Sing_X$ is a point with a one dimensional orbit. In this case, by Lemma~\ref{L: Normal spaces of directions}, Part \eqref{L: Normal spaces of directions-circle}, we have that $S_x^\perp\cong S^2$. Furthermore, the isotropy group $G_x$ is also $1$-dimensional. To obtain the local structure of $X^\ast$ around $x^\ast$, namely, $K((S_x^\perp)^\ast)$, the classification of the effective continuous actions of $1$-dimensional compact Lie groups on $S^2$ is needed. By Theorem~\ref{T: For dim at most 4, every topol coho 1 action is smooth} 
this action can be assumed to be a smooth action. From \cite[Section~2.3]{Hoe} we get that this action of $G_x$ on $S^2$ is equivalent to the action of $\mathrm{SO}(2)$, by rotating $S^2$ with respect to the north-south poles. The group diagram of this cohomogeneity one action is $(G(m,n), G(m, n), G(m, n), \{1\})$ and the orbit space $(S_x^\perp)^\ast$ is a closed interval. Therefore, $x^\ast$ is a boundary point of $X^\ast$ and the local structure of $X^\ast$ around such orbits is as in Figure~\ref{FIG_Structure_of_RC}.

\begin{figure}[ht]
\centering

\begin{subfigure}[t]{0.45\textwidth}
\centering
\begin{tikzpicture}[rotate=90,scale = 0.6]
\fill[color=gray!15!white] (-1.99,3.98) -- (0,0) -- (1.99,3.98)-- cycle;
\draw [] (-1.99,3.98) -- (0,0) -- (1.99,3.98)-- cycle;
\draw[color = black, fill=black] (0,0) circle (2pt) node () [right]{$ \ \soso$};
\node at (0,1.1)  {$\ \ \ \ x^{*}$};
\node at (1,1.8) [above] {$\ \ \ \ G(m, n)$};
\node at (-1,1.8) [below] {$\ \ \ \ G(p,q)$};
\node at (0,4) [left] {$\{1\}$};
\end{tikzpicture}
\caption{Structure of $B_{\varepsilon}(x^\ast)$ for a topologically regular fixed point.}
\label{FIG_Regular_Fixed}
\end{subfigure}
\hfill
\begin{subfigure}[t]{0.45\textwidth}
\centering
\begin{tikzpicture}[rotate=90,scale = 0.6]
\fill[color=gray!15!white] (-1.99,3.98) -- (0,0) -- (1.99,3.98)-- cycle;
\draw [] (-1.99,3.98) -- (0,0) -- (1.99,3.98)-- cycle;
\draw[color = black, fill=black] (0,0) circle (2pt) node () [right]{$\ G(m, n)$}; 
\node at (0,1.1)  {$\ \ \ \ x^{*}$};
\node at (1,1.8) [above] {$\ \ \ \ G(m, n)$};
\node at (-1,1.8) [below] {$\ \ \ \ G(m, n)$};
\node at (0,4) [left] {$\{1\}$};
\end{tikzpicture}
\caption{Structure of $B_{\varepsilon}(x^\ast)$ for a $1$-dimensional orbit of a topologically regular point.}
\label{FIG_Structure_of_RC}
\end{subfigure}
\caption{Neighborhoods in $X^\ast$ around  orbits of topologically regular points.}\label{F: Neighborhoods}
\end{figure}

\subsection{Finite Isotropy Groups} 

Let $x\in X\smallsetminus \Sing_X$ and $G_x$ be its isotropy group. Assume that $G_x$ is finite. Then, the orbit $G(x)$ is $2$-dimensional, and the normal space of directions to this orbit, $S_x^{\perp}$, is  $1$-dimensional. Therefore, $S_x^{\perp}\cong S^1$. 

To describe the slice and local picture of $X^\ast$ around $x^\ast$, we need to examine the effective actions of finite groups on $S^1$, preserving the orientation. Since $G_x$ is a closed subgroup (as it is finite) of the Lie group $\soso$, it is a compact Lie group by Cartan's theorem. Since we are assuming it preserves the orientation of $S^1$, it is conjugate to a closed subgroup of $\mathrm{SO}(2)$ (see \cite[Section~4.1]{Ghy}). Thus we conclude that $G_x$ is a cyclic group acting by rotations on $S^1$.

Hence, the local picture of the orbit space $X^\ast$ around $x^\ast$ is a $2$-disk contained in the interior of $X^\ast$. All of the points contained in this disk have trivial isotropy except  for $x^\ast$, which has $G_x$ as isotropy. Therefore, $x^\ast$ is an isolated interior point.

\subsection{Orbit types}

We have different orbit types according to the possible admissible isotropy groups of the action. In the following we recall those that already appear in the manifold case, and define a new orbit type to account for the presence of topologically singular points.

As in the manifold case, an orbit with trivial isotropy group is  called a \textit{principal orbit} or \textit{$P$-orbit}. We denote the set of points in $X$ whose orbits are $P$-orbits by $P$. By Theorem \ref{THM:Principal-orbit-theorem}, $P$ is an open and dense subset of $X$ consisting entirely of topologically regular points. Points having orbits with non-trivial finite isotropy are topologically regular points, and therefore, the concept of local orientation is well defined: An orbit with non-trivial finite isotropy acting by preserving the local orientation is called an \textit{exceptional orbit} or \textit{$E$-orbit}. The subset of $X$ of points lying on $E$-orbits is denoted by $E$. Each $E$-orbit is isolated from other $E$-orbits. The orbit of a topologically regular point whose isotropy is isomorphic to the circle is called a \textit{circular orbit} or \textit{$C$-orbit}. We denote the set of points on $C$-orbits by $C$. Note that $C$ consists entirely of topologically regular points. The set of topologically regular points which are fixed points of the action (that is, having isotropy $\soso$) is denoted by $RF$. By definition $RF$ consists only of topologically regular points. Finally we relabel the set of topologically singular points (which by our previous analysis are fixed points of the action) by $SF$.

We collect all of this information in Table \ref{TBL:orbit-types}. \hfill


\begin{table}[ht]
\resizebox{\textwidth}{!}{%
\begin{tabular}{|c|c|c|c|c|}
\hline
\textbf{Orbit Type} & \textbf{Notation} & \textbf{Isotropy} & \textbf{Space of Directions} & \textbf{Comments} \\ \hline
Principal & $P$ & $\{1\}$ & \multirow{6}{*}{$S^3$} & Interior point in $X^\ast$. \\ \cline{1-3} \cline{5-5} 
\multirow{2}{*}{Exceptional} & \multirow{2}{*}{$E$} & $\mathbb{Z}_k$ &  & Interior point in $X^\ast$. \\
 &  & (contained in some $G(m,n)$) &  & Isolated from other $E$ orbits. \\ \cline{1-3} \cline{5-5} 
Circular & $C$ & $G(m,n)$ &  & Boundary point in $X^\ast$ \\ \cline{1-3} \cline{5-5} 
Regular & \multirow{2}{*}{$RF$} & \multirow{2}{*}{$\soso$} &  & Lateral isotropies $G(m,n)$ and $G(p,q)$ \\
Fixed Point &  &  &  & with $mq-np=\pm 1$ \\ \hline
Singular & \multirow{2}{*}{$SF$} & \multirow{2}{*}{$\soso$} & $L(q,p)$ & Lateral isotropies $G(a,b)$ and $G(c,d)$ \\
Fixed Point &  &  & ( $(q,p)\neq (1,0), (0,1)$ ) & with $ad-bc\neq 0,\pm 1$ \\ \hline
\end{tabular}%
}
\caption{Orbit types of a $\soso$-action on $X$}
\label{TBL:orbit-types}
\end{table}


In the following proposition we summarize our previous discussions and state immediate consequences obtained by putting together the analysis of the local structure of $X^\ast$ around orbits of topologically singular points and the analysis at topologically regular points obtained in \cite{OrlRayI, OrlRayII}. 

\begin{prop}
\label{P:Orbit-space-structure}
Let $X$ be a closed, orientable Alexandrov $4$-space with an effective, isometric $\soso$-action. Then the following hold:
\begin{itemize}
\item[(1)] The orbit space $X^\ast$ is a $2$-manifold (possibly) with boundary.
\item[(2)] The interior of $X^\ast$ consists of $P$-orbits and a finite number of $E$-orbits. 
\item[(3)] Each connected component of the boundary of $X^\ast$ consists of $C$, $RF$ and $SF$ orbits. 
\end{itemize}
\end{prop}

\section{Equivariant classification}
\label{S_Equivariant classification}

As done for smooth $\soso$-actions on $4$-manifolds (see \cite{OrlRayI,OrlRayII}) and $\mathrm{SO}(2)$-actions on $3$-manifolds (see \cite{Nun}), we may endow the orbit space $X^\ast$ with a set of ``weights" that encode its topology, the isotropy information of the action, and the topological regularity of the points in the orbit. In this section we prove that this set of weights characterizes the action. We list these invariants:
\begin{enumerate}[(i)]

\item For a closed, oriented Alexandrov $4$-space $X$ with an effective $\soso$-action by isometries we have by Proposition~\ref{P:Orbit-space-structure} that the orbit space $X^\ast$ is an oriented $2$-manifold of a certain genus possibly with boundary. The integer $g\geq 0$ denotes the genus of $X^\ast$. 
 
\item Fix an orientation for $\soso$. Recall that the $2$-manifold $X^\ast$ inherits an orientation from $X$ via the orbit projection map and vice versa. We denote such an orientation of $X^\ast$ by the symbol $\varepsilon$. Hence, $\varepsilon$ has two different possible values.
 
\item The integer $s\geq 0$ denotes the number of boundary components of $X^\ast$ which consist only of $C$-orbits. To each connected component of this type we associate the weight $\left\langle p_i,q_i\right\rangle$ corresponding to the isotropy subgroup $G(p_i,q_i)$ of the corresponding circular orbits.

\item The integer $t\geqslant 0$  is the number of boundary components of $X^\ast$ which have non-empty intersection with $RF^\ast\cup SF^\ast$. We label these boundary components from $1$ to $t$. On the $\ell$-th boundary component of this type, there are $r_\ell \geqslant 2$  fixed points, and the orbits  between two such fixed points are circular orbits with isotropy $G(a_{\ell,w},b_{\ell,w})$. We set the  weight $f_{\ell,w}$ to be the determinant of the isotropy groups $G(a_{\ell,w},b_{\ell,w})$ and $G(a_{\ell,w+1},b_{\ell,w+1})$, i.e. $f_{\ell,w} =a_{\ell,w}b_{\ell,w+1}-a_{\ell,w+1}b_{\ell,w}$, for $1\leq w\leq r_{\ell}-1$, and $f_{\ell,r_{\ell}} =a_{\ell,r_{\ell}}b_{\ell,1}-a_{\ell,1}b_{\ell,r_{\ell}}$. Therefore, for a fixed topologically regular point we have $f_{\ell,w} = \pm 1$, and for a fixed topologically singular point we have $f_{\ell,w}  \neq \pm 1 $, and $f_{\ell,w}  \neq 0$. Hence, we associate to the $\ell$-th boundary component the sequence of weights 
\begin{align}
\big( (a_\ell,b_\ell),f_\ell \big)= \big(&(a_{\ell,1},b_{\ell,1}),f_{\ell,1},(a_{\ell,2},b_{\ell,2}),\ldots ,\\
&(a_{\ell,(r_\ell-1)},b_{\ell,(r_\ell-1)}), f_{\ell,(r_\ell-1)}, (a_{\ell,r_\ell},b_{\ell,r_\ell}),f_{\ell,r_\ell} \big).\nonumber
\end{align}
If $r_{\ell}=2$, for some $1\leq \ell\leq t$, then we require that $f_{\ell, 1}=-f_{\ell, 2}$. This implies that for $r_{\ell}=2$, both fixed points are topologically singular, or they both are topologically regular.

\item Let $k\geqslant 0$ denote the number of exceptional orbits. To each exceptional orbit we associate the weight $(\alpha_l;\gamma_{l,1},\gamma_{l,2})$, the so-called \textit{oriented Seifert invariants} of the orbit (see \cite[Page 93]{OrlRayII} for the precise definition). 
\item In the case that $C\cup RF\cup SF=\emptyset$ we further associate another invariant. Consider $E^\ast=\{x_i^\ast\}_{i=1}^k$, and let $D_i^\ast$ be disjoint closed $2$-disks in $X^\ast$ centered at $x_i^\ast$ such that $D_i^\ast\setminus \{x_i^\ast\}\subset P^\ast$ for all $1\leq i \leq k$. We arbitrarily choose a $P$-orbit $x_0^\ast\in X^\ast\setminus \bigcup_{i=1}^k D_i^\ast$ and let $D_0^\ast$ be a small $2$-disk centered at $x_0^\ast$ fully contained in $P^\ast$.  Let $\chi:\bigcup_{i=0}^k \partial D_i^\ast \to X$ be a cross-section to the restriction of the action to  $\bigcup_{i=0}^k \partial D_i^\ast$. By standard Obstruction Theory this cross-section can be extended to $X^\ast \setminus \bigcup_{i=0}^k \mathrm{int}(D_i^\ast)$ and  the obstruction to extending it to $X^\ast\setminus \bigcup_{i=1}^k \mathrm{int}(D_i^\ast)$ is an element 
\[
(b_1,b_2)\in H^2\left(X^\ast, \bigcup_{i=1}^k D_i^\ast; \mathbb{Z}\oplus \mathbb{Z} \right)\cong \mathbb{Z}\oplus \mathbb{Z}.
\] 
We refer the reader to \cite[1.3]{OrlRayII} for a more detailed exposition of this invariant. 
\end{enumerate}

In sum, to every orbit space $X^\ast$ of an effective and isometric $\soso$-action on $X$ we associate the following set of invariants,
\begin{equation}
\label{EQ:INVARIANTS}
\big\{(b_1,b_2);\varepsilon;g; \big\{ \left\langle p_i,q_i\right\rangle \big\}_{i=1}^s; \big\{ \big( (a_{\ell},b_{\ell}),f_{\ell}\big) \big\}_{\ell = 1}^{t} ; \big\{ (\alpha_j;\gamma_{j,1},\gamma_{j,2}) \big\}_{j=1}^k \big\}.
\end{equation}


In the following we show that, as in the smooth case, this set of weights classifies not only the space $X$ up to equivariant homeomorphism, but also the action up to orientation. Some remarks are in order. 

\begin{rem}
If one of the sets $C$, $RF\cup SF$ or $E$ is empty, we denote it on the set of invariants by --.
\end{rem}

\begin{rem}
Observe that the set of invariants \eqref{EQ:INVARIANTS} reduces to the set of invariants given by Orlik and Raymond \cite{OrlRayI, OrlRayII} in the case that $X$ is homeomorphic to a topological manifold. Indeed, in this case,  $SF=\emptyset$, and $f_{\ell,w}=\pm 1$ for all values of $\ell$ and $w$, making the invariant $f_{\ell,w}$ superfluous.
\end{rem}

\begin{rem}
A boundary component of $X^\ast$ may contain an odd number of topologically singular fixed points.
In particular, it may contain a single topologically singular fixed point, as long as the total number of fixed points is at least three (see Example~\ref{EX: weighted projective space}, with $(m_1,n_1) = (1,0)$, $(m_2,n_2) = (0,1)$, and $(m_3,n_3) = (m,1)$, for $m\neq 0, \pm 1$.). 
This stands in contrast with the case of circle actions on $3$-dimensional closed Alexandrov spaces (cf. \cite[Lemma 3.3]{Nun} )
\end{rem}


\begin{figure}[h!]
\includegraphics[scale=0.5]{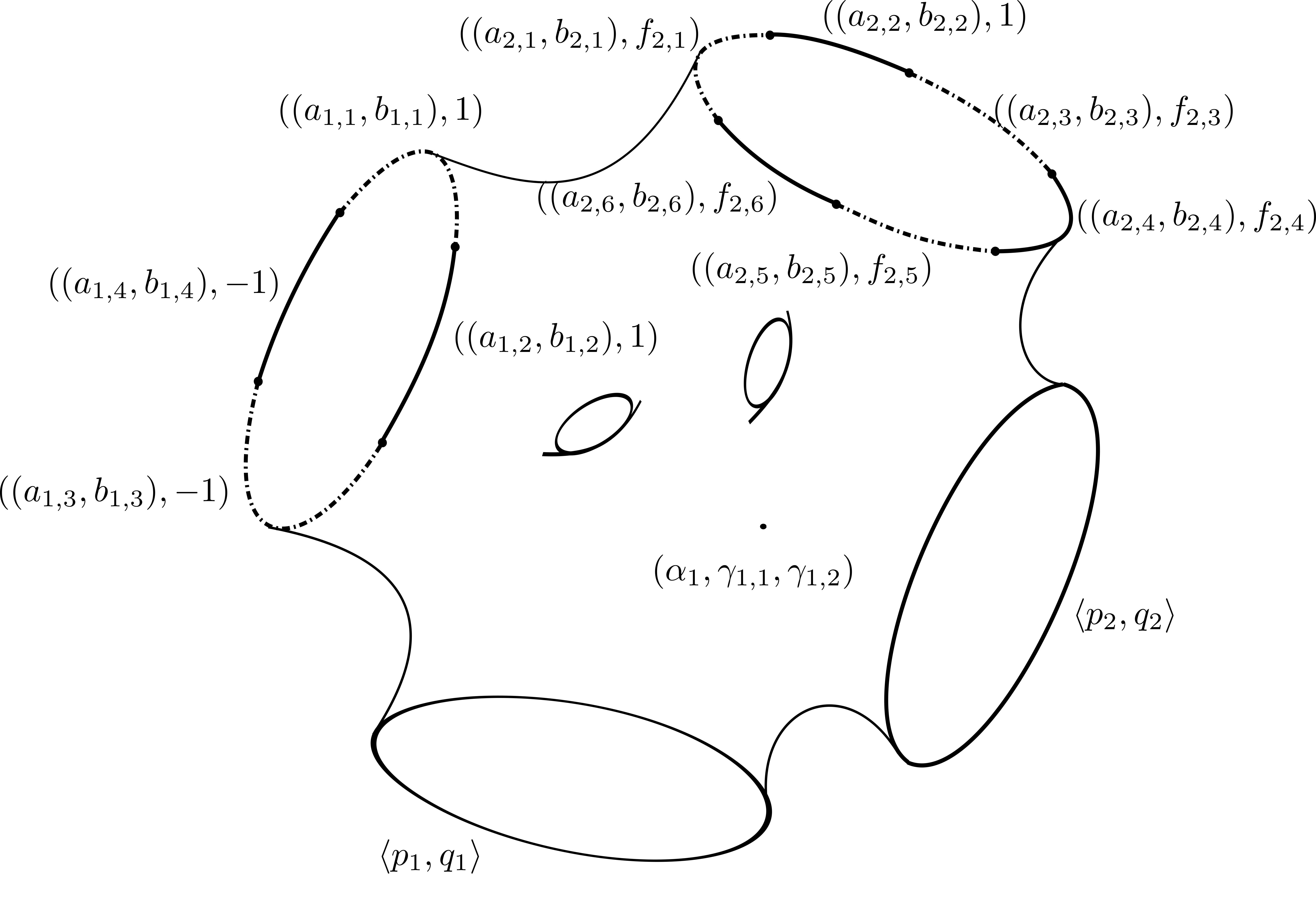}
\caption{Example of a weighted orbit space of an effective and isometric $\soso$-action on a closed Alexandrov $4$-space.}
\label{FIG:GENERIC_ORBIT_SPACE}
\end{figure} 



Let us now recall the definition of isomorphism between weighted orbit spaces.

\begin{definition}
Let $X$ and $Y$ be two closed, oriented Alexandrov $4$-spaces carrying effective and  isometric $T^2$-actions. We  say that the orbit spaces $X^\ast$ and $Y^\ast$ are \textit{isomorphic} if there exists an orientation preserving 
homeomorphism $f\colon X^\ast\to Y^\ast$, such that for each $x\in X^\ast$, the weight associated to $f(x^\ast)$ is the same as the weight associated to $x^\ast$. 
\end{definition}

\subsection{Cross-sectioning theorems}

As in the manifold case, a key step in achieving the equivariant classification of effective and isometric $\soso$-actions, is to show the existence of a cross-section to the quotient map $\pi: X\to X^\ast$ in the absence of exceptional orbits.

This has been done in  \cite[Appendix]{OrlRayI} for the more general family of  $4$-dimensional cohomology manifolds with a continuous $\soso$-action. As pointed out in Section~\ref{S:Local structure around topological regular points}, an orientable $4$-dimensional Alexandrov space $X$ with an effective and isometric $\soso$-action is a cohomology manifold (see also \cite[Corollary~8.7]{Mitsuishi}). However, we present a more detailed proof here given that we have a clearer picture of the local structure of $X$. 


We  first consider the simpler case when the orbit space is homeomorphic to a $2$-disk and proceed in a similar fashion to the proof of Theorem~1.10 in \cite{OrlRayI} for smooth closed $4$-manifolds with a smooth $\soso$-action, by splitting the proof on the existence of the cross-section into several lemmas. More precisely, we show how we can define a cross-section on simple subsets (in terms of the topology and isotropy information) and then show it is possible to extend any given cross-section on such subsets to a full cross section.

\begin{lem}[cf. \protect{\cite[Lemma~1.6]{OrlRayI}}]\label{L:Cross-section circle orbits}
Let $X$ be a compact Hausdorff space with an effective continuous $T^2$-action. Assume that the orbit space is homeomorphic to $I\times I$, where the orbits on the arc $[0,1]\times\{0\}$ all have isotropy $G(m,n)$. Then there exist a cross-section $\rho\colon X^\ast\to X$. Moreover, assume that over a connected subset $A^\ast$ of $J^\ast = \{0\}\times I \cup I\times\{1\}\cup\times\{1\}\times I$ we have defined a cross-section. Then this cross-section can be extended to a cross-section over all of $X^{*}$. 
\end{lem}

\begin{proof}
Let $G(a,b)$ be a circle subgroup of $T^2$ such that $an - bm = 1$, i.e. such that $T^2 = G(m,n)\times G(a,b)$. Consider $Y = X/G(m,n)$. Then $G(a,b)$ acts on $Y$ freely, and we have a principal circle fibration $Y\to X^\ast$, with fiber $G(a,b)$.  Since $X^\ast$ is contractible then $Y$ is equivariantly homomorphic to $X^\ast\times G(a,b)$ and there exist a cross-section $\rho_1\colon X^\ast\to Y$. Next, we observe that over $X^\ast\setminus I\times\{0\}$, we have a principal $G(m,n)\times G(a,b)$-bundle. Again, since $X^\ast\setminus I\times\{0\} $ is contractible, we conclude that over $X^\ast\setminus I\times\{0\} $ the bundle is trivial, i.e. the total space is $(X^\ast\setminus I\times\{0\} )\times G(m,n)\times G(a,b)$, and in particular we have a cross-section $\rho_2\colon Y \setminus (I\times\{0\} \times G(a,b) ) \to (X^\ast\setminus I\times\{0\} )\times G(m,n)\times G(a,b)\subset X$. We can extend this cross-section to all of $Y$ since over $I\times\{0\} \times G(a,b)$ we have only the fixed points of the $G(m,n)$-action on $X$. Thus, combining the cross-sections $\rho_1\colon X^\ast\to Y$ and $\rho_2\colon Y\to X$ we obtain the desired cross-section $\rho\colon X^\ast\to X$. We also conclude that $X$ is equivariantly homeomorphic to $X^\ast\times G(m,n)\times G(a,b)/\sim$, where  $((t,0),\zeta_1,\xi)\sim ((t,0),\zeta_2,\xi)$ for $t\in I$, $\zeta_1,\zeta_2\in G(m,n)$, and $\xi\in G(a,b)$.

Assume that we are given a cross-section  $\sigma\colon A^\ast\to X$. Since $X$ is equivariantly homeomorphic to $X^\ast\times G(m,n)\times G(a,b)/\sim$, we have a cross-section $\sigma\colon A^\ast\to X^\ast\times G(m,n)\times G(a,b)/\sim$. We prove that we can extend this cross-section to all of $X^\ast$. We fix $Z = X^\ast\times G(m,n)/\sim$, where $((t,0),\zeta_1)\sim  ((t,0),\zeta_2)$ for $t\in I$ and $\zeta_1, \zeta_2\in G(m,n)$. Observe that $Z$ is the orbit space of the $G(a,b)$-action on $X^\ast\times G(m,n)\times G(a,b)/\sim$. Thus, we have two projection maps $\pi_1 \colon X^\ast\times G(m,n)\times G(a,b)/\sim = X \to Z$ and $\pi_2 \colon Z\to X^\ast$. Moreover, the cross-section $\sigma\colon A^\ast \to X^\ast\times G(m,n)\times G(a,b)/\sim$ is of the form $\sigma(a^\ast) = [a^\ast,\zeta(a^\ast), \xi(a^\ast)]$. Observe that we have a cross-section $\sigma_2\colon A^\ast\to Z$ for $\pi_2$ defined as $\sigma_2(a^\ast) = [a^\ast,\zeta(a^\ast)]$. We fix $A_1 = \sigma_2(A^\ast)\subset Z$, and $A= \sigma(A^\ast)\subset X^\ast\times G(m,n)\times G(a,b)/\sim$. We obtain a cross-section $\sigma_1\colon A_1\to A$ for $\pi_1$, given by $\sigma_1[a^\ast,\zeta(a^\ast)] = [a^\ast,\zeta(a^\ast),\xi(a^\ast)]$. 

We show first that we can extend $\sigma_2$ to $J^\ast$. We assume that $A^\ast$ is  a proper subset of $J^\ast$. After removing the fixed points $\{(0,0), (0,1)\}$ from $J^\ast$, and possibly $A^\ast$, we get two spaces $J_1^\ast\subset J^\ast$ and $B^\ast\subset A^\ast$ which have the same homotopy type. Moreover, we have a circle fibration $\pi_2\colon \pi_2^{-1}(J_1^\ast)\to J_1^\ast$, since the action of $G(m,n)$ over $\pi_2^{-1}(J_1^\ast)$ is free. From the fact that $J_1^\ast$  has the same homotopy type as $B^\ast$, we conclude	 that $H^2(J_1^\ast,B^\ast;\ZZ)=0$, and thus the obstruction to extend $\sigma_2$ from $B^\ast$ to $J_1^\ast$ vanishes (see \cite[pp. 415--419]{Hatcher}). Therefore,  we can extend the map $\sigma_2\colon B^\ast\to Z$ to $J^\ast_1$. By adding back the fixed two fixed points to $J_1^\ast$, we get a cross-section $\sigma_2\colon J^\ast\to Z$ extending the  cross-section $\sigma_2\colon A^\ast\to A_1$. 

We show that we can extend this cross-section $\sigma_2\colon J^\ast \to Z$ to a cross-section $\sigma_2\colon X^\ast\to Z$. Consider the point $p^\ast = (1/2,0)\in X^\ast$. Observe that $X^\ast$ is homeomorphic to the cone of $J^\ast$ with vertex $p^\ast$, denoted by $C(p^\ast,J^\ast)$.  For $x^\ast\in J^\ast$ and $1\geqslant t>0$ we set $\sigma_2(x^\ast,t) = [(x^\ast,t),\zeta(x^\ast)]\in Z$, and at the vertex, $\sigma_2(p^\ast) = [p^\ast,\zeta]$ for any $\zeta\in G(m,n)$. 

Now we prove we can extend the cross-section $\sigma_1\colon A_1\to X$ to a cross-section $\sigma_1\colon Z\to X$. Recall that $\pi_1\colon X\to Z$ is a fibration with fiber a circle, since the action of $G(a,b)$ on $X$ is free. We point out that $A_1 = A^\ast\times G(m,n)/\sim$ has the same homotopy type as $Z = X^\ast\times G(m,n)/\sim$. From this it follows that $H^2(Z,A_1;\Z) =0$, and thus the only obstruction to extending  the given cross-section $\sigma_1\colon A_1\to  X^\ast \times G(m,n)\times G(a,b)/\sim$  to the whole of $Z$ vanishes (see \cite[pp. 415--419]{Hatcher}). Composing $\sigma_2$ with $\sigma_1$ we obtain  a cross-section $\sigma = \sigma_1\circ\sigma_2 \colon X^\ast\to X$ extending the given cross-section $\sigma\colon A^\ast\to X$.

\end{proof}

\begin{rem}\label{R: Cross-section with one corner fixed point}
Observe that in the previous proof, we can assume that the corner point $(0,0)\in X^\ast$ is a fixed point of the $T^2$ action on $X$, and we still recover the same conclusion: given a cross-section on $A^\ast$, we can extend this cross-section to $X^\ast$.
\end{rem}

\begin{lem}[cf. \protect{\cite[Lemma~1.8]{OrlRayI}}]\label{L:Section around fixed points}
Let $X$ be a compact Hausdorff space with an effective continuous action of $\soso$,  and assume that $X^\ast = [-1, 1] \times I$, where the orbits in  $[-1, 0)\times \{0\}$ have isotropy subgroup $G(m,n)$, the orbits in $(0, 1] \times \{0\}$ have isotropy subgroup $G(p,q)$ such that $mp-nq \neq 0$, the point $(0, 0)\in X^\ast$ is a fixed point,  and all other points correspond to principal orbits. Then there is a cross-section to this action. Moreover any cross-section on a connected subset $A$ of the arc $ J^\ast = \{-1\} \times I \cup [-1,1]\times \{1\}  \cup \{1\} \times I$ may be extended to a cross-section over $X^\ast$.
\end{lem}

\begin{proof}
We split the orbit space $X^\ast$ into two components: $X_1^\ast = [-1,0]\times I$, and $X_2^\ast = [0,1]\times I$. Assume without loss of generality that $A^\ast\cap X_1\neq \emptyset$. We then apply Remark~\ref{R: Cross-section with one corner fixed point} to obtain a cross-section $\sigma_1\colon X_1^\ast\to X$ extending the given cross-section $\sigma\colon A^\ast \cap X_1^\ast \to X_1$. We then apply Remark~\ref{R: Cross-section with one corner fixed point} to the cross-section $\sigma_2\colon \{0\}\times I \cup A^\ast\cap X_2^\ast \to X$, where $\sigma_2(0,t) = \sigma_1(0,t)$ for $t\in I$, and $\sigma_2(a^\ast ) = \sigma(a^\ast)$ for $a\in A^\ast\cap X_2^\ast$. Thus we obtain a continuous cross-section $\sigma\colon X^\ast \to X$ by setting $\sigma(x^\ast )= \sigma_1(x^\ast)$ for $x^\ast\in X_1^\ast$, and $\sigma(x^\ast )= \sigma_2(x^\ast)$ for $x^\ast\in X_2^\ast$
\end{proof}

Another lemma that we need to prove the existence of the cross section  is as follows:
\begin{lem}[\protect{\cite[Lemma~1.7] {OrlRayI}}]
\label{L:Section on Cylinder}
Suppose that $\pi\colon X\to X^{\ast}$ is the orbit map of a $\soso$-action on a compact Hausdorff space and assume that $X^{\ast} = \Ss^{1} \times  I$, where $\Ss^{1} \times \{0\} $ corresponds to orbits with isotropy group $G(m, n) $ with $\mathrm{gcd}(m, n)= 1$ and all other points to principal orbits. Then the map $\pi$
has a cross-section. Moreover any cross-section on $\Ss^{1} \times \{0\}$ may be extended to a cross-section over $X^{\ast}$.
\end{lem}

\begin{rem}
We observe that the hypotheses of both lemmas only  require the space to be Hausdorff, a condition which is satisfied by Alexandrov space since they are metric spaces.
\end{rem}

\begin{prop}\label{P: X^* a disk then there is a cross-section}
Let $X$ be a closed, orientable Alexandrov $4$-space with an effective, isometric $T^2$-action. If the orbit space $X^\ast$ is homeomorphic to a $2$-disk and $E=\emptyset$, then there exists a cross-section.
\end{prop} 

\begin{proof}
We follow the same procedure as done in the proof of \cite[Theorem~1.10]{OrlRayI}. If $RF\cup SF=\emptyset$, then $X$ is homeomorphic to a $4$-manifold and the torus action in question has no fixed points. In this case, the result is exactly Case 1 of \cite[Theorem~1.10]{OrlRayI}. 

Therefore, we assume that $RF\cup SF\neq \emptyset$. The strategy to deal with this case follows along the lines of Case 2 of \cite[Theorem~1.10]{OrlRayI}. Let $RF^\ast\cup SF^\ast=\{ x_1^\ast,\ldots, x_r^\ast\}$. We split $X^\ast$ as $Y^\ast\cup D_1^\ast\cup\cdots \cup D_{r}^\ast$, where $Y^\ast$ is a closed $2$-disk contained in the interior of $X^\ast$, and each $D_{i}^\ast$ is homeomorphic to $[-1,1]\times I$, and satisfies:  $D_i^\ast\cap \left(RF^\ast\cap SF^\ast \right)=\{x_i^\ast\}$, and $D_i^\ast\cap Y^{\ast}$ corresponds to $[-1,1]\times \{1\}$. The set $D_i^\ast\cap D_{i+1}^\ast$ corresponds to $\{1\}\times I$ and the distribution of isotropies on $D_i^\ast$ is exactly as in Lemma \ref{L:Section around fixed points} (identifying $D_i^\ast\cap \partial X^\ast$ with $[-1,1]\times\{0\}$ and $x_i^\ast$ with $\{0\}\times\{0\}$).  
Since $Y^\ast$ does not contain orbits with non-trivial isotropy,  the restriction of the quotient map $\pi:\pi^{-1}(Y^\ast) \to Y^\ast$ is a principal $T^2$-bundle. Furthermore, since $Y^\ast$ is contractible, this principal bundle is trivial. Thus, there exists a cross-section $\eta:Y^\ast\to \pi^{-1}(Y^\ast)$. 

By considering the restriction of $\eta$ to $D_i^\ast\cap Y^{\ast}$  we fall into the hypotheses of Lemma \ref{L:Section around fixed points}. Hence, by this result, we can extend $\eta$ to $Y^\ast\cup D_1^\ast$. Slightly abusing the notation we do not rename the cross-section. Similarly, by Lemma \ref{L:Section around fixed points} and Remark~1.9 in \cite{OrlRayI}, the restriction of $\eta$ to $D_2^\ast\cap (Y^\ast\cup D_1^\ast)$ can be extended to a cross-section on $Y^\ast\cup D_1^\ast\cup D_2^\ast$. Proceeding inductively, it is clear that there exists an extension of $\eta$ to $X^\ast$, thus proving the result. 
\end{proof}

\subsection{Equivariant classification}

\begin{thm}
\label{T:Equivalent iff orbit spaces are isomorphic}
Let $X_1$ and $X_2$ be two closed, oriented Alexandrov $4$-spaces each with an effective, isometric $\soso$-action. Then $X_1$ and $X_2$ are equivariantly homeomorphic  if and only if $X_1^\ast$ and $X_2^\ast$ are isomorphic.
\end{thm}

\begin{proof}
The ``only if" part is clear. Therefore, we focus on showing that if $X_1^\ast$ is isomorphic to $X_2^\ast$, then $X_1$ is equivariantly homeomorphic to $X_2$. To this end, we begin by observing that if $SF_i=\emptyset$, then $X_i$ are homeomorphic to $4$-manifolds. In this case the proof is as in \cite{OrlRayI, OrlRayII}. Hence, in the following we assume $SF_i\neq \emptyset$. Suppose there is an isomorphism $h^\ast\colon X_1^\ast\to X_2^\ast$  of orbit spaces. We also consider the projection maps $\pi_1\colon X_1\to X_1^\ast$ and $\pi_2\colon X_2\to X_2^\ast$. Assume there are exceptional orbits $x_1^\ast,\ldots,x_k^\ast$ in the interior of $X_1^\ast$.  Since all the topologically singular points lie in the boundary of $X_1^\ast$, this exceptional orbits are topologically regular points, and thus via the slice theorem we can find small tubular neighborhoods $D_s\subset X$ such that $D_s^\ast\smallsetminus \{x_s^\ast\}$  and $h^\ast(D_s^\ast)\smallsetminus \{h^\ast(x^\ast)\}$ contain only principal orbits. Since $H^2(\partial D_s^\ast,\ZZ\oplus\ZZ)=0$ and $H^2(\partial h^\ast(D_s^\ast),\ZZ\oplus\ZZ)=0$, we have  cross-section  on the boundaries  of $D_s^\ast$ and $h^\ast(D^\ast_s)$. Furthermore, assume the boundary of $X_1^\ast$ consist of $m$ connected components. Set $U_j^\ast$ to be a closed annular neighborhood of the $j$-th component. 

Then for $X_{1,1}=X_1\setminus (D_1^\ast\cup\cdots\cup D_k^\ast\cup U_1^\ast\cup\cdots\cup U_m^\ast)$, and $X_{1,2} =  X_2\setminus (h^\ast(D_1^\ast)\cup\cdots\cup h^\ast(D_k^\ast)\cup h^\ast(U_1^\ast)\cup\cdots\cup h^\ast(D_m^\ast))$ from the  proof of Theorem~1.10 in \cite{OrlRayI}, the sections on each $\partial D_s^\ast$  can be extended to a section $\sigma_1\colon X_{1,1}^\ast\to X_{1,1} = \pi_1^{-1} (X^\ast_{1,1})$. The same is true for the sections over the boundaries of $h^\ast(D_s^\ast)$; we can extend them to a cross-section $\sigma_2\colon X_{1,2}^\ast\to X_{1,2} = \pi_2^{-1}( X_{1,2}^\ast)$. Since the action of $T^2$ over $X_{1,1}$ and $X_{1,2}$ is free, we conclude that $X_{1,1}$ is equivariantly homeomorphic to $X_{1,1}^\ast\times T^2$, and $X_{1,2}$ is equivariantly homeomorphic to $X_{1,2}^\ast\times T^2$. Thus we have an equivariant homeomorphism $h\colon  X_{1,1}^\ast\times T^2 \to X_{1,2}^\ast\times T^2$ given by $h(x^\ast,\xi,\zeta) = (h^\ast(x^\ast), \xi,\zeta)$.

The sections $\sigma_1$ and $\sigma_2$ can be extended over each annulus $U_j^\ast$, and $h^\ast(U_j^\ast)$ by Theorem~1.10 in \cite{OrlRayI} and Proposition~\ref{P: X^* a disk then there is a cross-section}. Consider $x^\ast\in U_j^\ast$. Observe that for any $y\in \pi_1^{-1}(x^\ast)$, there exists and element $(\xi, \zeta)\in T^2$, such that $y = (\xi, \zeta)\sigma_1(x^\ast)$; the same is true for $h(x^\ast)\in h^\ast(U_j^\ast)$. Set $X_{2,1}^\ast = X_1^\ast\smallsetminus (D_1^\ast\cup \cdots D_k^\ast)$, $X_{2,1} = \pi_1^{-1}(X_{2,1}^\ast)$, and $X_{2,2} = X_2\smallsetminus (D_1^\ast\cup \cdots D_k^\ast)$, $X_{2,2} = \pi_2^{-1}(X_{2,2}^\ast)$. We can extend the equivariant homeomorphism $h$, to an equivariant homeomorphism $h\colon X_{2,1}\to X_{2,2}$ as follows: for $(\xi,\zeta)\sigma_{1}(x^\ast)\in U_j= \pi_1^{-1}(U_j^\ast)$ we set $h((\xi,\zeta)\sigma_{1}(x^\ast)) =  (\xi,\zeta)\sigma_2(h^\ast(x^\ast))$.

To finish the proof, we have to extend the equivariant homeomorphism $h$ to each $D_s$. We recall that the actions on each $D_s$ are completely understood (see \cite[Section~1.1]{OrlRayII}). In this sense, there is a unique equivariant way to attach each of them to $X_{2,1}$ and $X_{2,2}$. Therefore, we have an equivariant homeomorphism $h\colon X_1\to X_2$ lifting the given isomorphism $h^\ast$. 
\end{proof}


\section{Construction of an Alexandrov metric}
\label{S:Construction_Alexandrov_metric}

We have so far showed that two closed, oriented Alexandrov $4$-spaces with an effective and isometric $T^2$-action which have the same set of invariants must be equivariantly homeomorphic. To complete the equivariant classification we show that given an arbitrary set of invariants as in \eqref{EQ:INVARIANTS}, there exists a closed orientable Alexandrov $4$-space $X$ and an effective and isometric $T^2$-action on $X$ which has precisely this set of invariants. More precisely, we assume that we are given a compact $2$-dimensional topological manifold endowed with a set of invariants as in Section~\ref{S_Equivariant classification}. Then  using a result of Haefliger and Salem 
we construct a smooth orbifold with a smooth $\soso$-action giving rise to the given weighted orbit space.   From this we construct an invariant Riemannian orbifold metric with curvature bounded below. 

Since it is a central tool in this section, for the sake of completeness, we recall the result of Haefliger and Salem. 

\begin{prop}[\protect{\cite[Proposition~4.5]{HaSa}}]
\label{T:Haefliger-Salem-orbifolds}
Let $W$ be a paracompact  space and $\{V_i\}_{i\in I}$ an open covering of $W$ such that for each $i\in I$, there exists a $T^{n}$-orbifold $X_i$ with orbit space equal to $V_i$. Assume the following compatibility condition: Let $\pi_i:X_i\to V_i$ be the projection map of each $T^{n}$-action. For each $w\in V_i\cap V_j$, there is an open neighborhood $V$ of $w$ in $V_i\cap V_j$ and a $T^n$-equivariant diffeomorphism of $\pi_i^{-1}(V)$ onto $\pi_j^{-1}(V)$ over $V$. Then to these compatible local data there is an associated  cohomology class in $H^{3}(W,\mathbb{Z}^n)$ whose vanishing is equivalent to the existence of a $T^n$-orbifold $X$ with orbit space $W$, such that $\pi^{-1}(V_i)$ is locally equivalent to $X_i$, where $\pi:X\to W$ is the projection map. 
\end{prop}

In the following definition we make explicit the restrictions that the a given set of weights needs to satisfy in order to be considered the set of invariants of an $\soso$-action on a $4$-dimensional orbifold.

\begin{definition}
Consider  $X^\ast$, a compact topological $2$-manifold (possibly with boundary) which has associated to it the  following set of weights:
\begin{equation}\label{EQ:GIVEN_SET_OF_INVARIATNS}
\big\{(b_1,b_2);\varepsilon;g; \big\{ \left\langle p_i,q_i\right\rangle \big\}_{i=1}^s; \big\{ \big( (a_{\ell},b_{\ell}),f_{\ell}\big) \big\}_{\ell=1}^{t} ; \big\{ (\alpha_j;\gamma_{j,1},\gamma_{j,2}) \big\}_{j=1}^k \big\}.
\end{equation}

We say that $X^\ast$ is \emph{legally weighted} if the determinant of two adjacent weights, $(a,b)$ and $(c,d)$, in $X^\ast$ is non-zero, i.e. $ad-bc\neq 0$. 
\end{definition}

\begin{rem}
The condition of being legally weighted is required to get, at a fixed point, as space of directions a suspension of a spherical space. In other words to exclude $S^2\times S^1$ as space of directions (see \cite{Neu}).
\end{rem}

\begin{rem}
Let $X^\ast$ be homeomorphic to a $2$-disk with two fixed points: one, topologically regular and the other, topologically singular. In other words we have the following weight information:
\[
	\big\{(\varepsilon, 0,0,0,1);\big\{\big((a,b),\delta,(c,d),1\big)\big\}\big\},
\]
where $\delta\neq \pm 1$. By Section~\ref{S_Equivariant classification}, Item (iv), we  exclude such weighted orbit space.
\end{rem}

\begin{rem}\label{R: Orbit spaces of Alexandrov spaces are leg. weighted}
Observe that if $X$ is a closed orientable $4$-dimensional Alexandrov space with an effective $\soso$-action by isometries, then the orbit space $X^\ast$ with the set of invariants described in Section~\ref{S:Orbit space and Orbit types}, is a legally weighted space.
\end{rem}

The following theorem shows that legal sets of weights are realized by orbifolds. 

\begin{thm}\label{T: existence of orbifold}
Let $X^\ast$ be the legally weighted topological $2$-manifold with weights as in \eqref{EQ:GIVEN_SET_OF_INVARIATNS}. Then there exists an orbifold $\mathcal{O}_X$ with a $\soso$-action, such that $\mathcal{O}_X / \soso$ is isomorphic to $X^\ast$.
\end{thm}

\begin{proof}
We begin by pointing out that if $SF^\ast=\emptyset$ then the result follows directly from \cite{OrlRayI, OrlRayII}, and in this case $\mathcal{O}_X$ is in fact a smooth $4$-manifold. Thus, in the following we assume that $SF\neq \emptyset$.  To prove the result in this case we construct an open cover of $X^\ast$ that satisfies the hypothesis of \cite[Proposition~4.5]{HaSa}: we give an open cover $\{V^\ast_i\}$ of $X^\ast$, such that for each open set $V^\ast_i$, there exists an orbifold $V_i$ with a smooth $\soso$-action, projection map $\pi_i\colon V_i\to V_i/\soso$, such that $V_i^\ast$ is the orbit space of the action, and for $w^\ast\in V_i^\ast\cap V_{j}^\ast$, there exists an open neighborhood $V^\ast\subset V_i^\ast\cap V_j^\ast$ of $w^\ast$, and an equivariant diffeomorphism between $\pi_{i}^{-1}(V^\ast)$ and $\pi_{j}^{-1}(V^\ast)$.

Recall that $t+s$ is the total number of boundary components of $X^\ast$.  Let us write $\partial X^\ast = (\sqcup_{i=1}^q \partial X^\ast_i)\sqcup (\sqcup_{j=q+1}^{t+s}\partial X^\ast_j)$ for some $q\geq 0$ in such a way that for $q+1 \leqslant j \leqslant t+s$, the boundary component $\partial X^\ast_j$ contains  at least one $SF$-orbit, and for $1 \leqslant i \leqslant q$, the boundary  component $\partial X^\ast_i$  contains no $SF$-orbit (i.e., it contains only $C$- and $RF$-orbits). Here, we adopt the convention that if $C^\ast,RF^\ast=\emptyset$ then $q=0$ and $(\sqcup_{i=1}^q \partial X^\ast_i)=\emptyset$.

Consider the following open cover for $X^\ast$. Assume that $W_i^\ast$ and $U_{i}^\ast$,  $q+1\leq i\leq t+s$, are annular neighborhoods around  each boundary component with $SF^\ast$-points, such that $\bar{U}_{i}^\ast$  is properly contained in $W_i^\ast$, and $W_i^\ast\cap W_j^\ast=\emptyset$, for all $i\neq j$. Let us denote $Y^\ast =X^\ast\setminus {\bigcup_{i=q+1}^{t+s} \bar{U}_i^\ast} $. Note that  $Y^\ast$ is a $2$-surface of genus $g$ with boundary. We now divide each open neighborhood $W_i^\ast$ into smaller neighborhoods $Z^\ast_{ij}$, where each $Z^\ast_{ij}$ is required to contain only  a single fixed point, and $\bigcup_j Z_{ij}^\ast=W_i^\ast$ (see Figure~\ref{FIG_OPEN_COVER}). 


\begin{figure}
\centering
\def\svgwidth{0.50\columnwidth} 
\includegraphics[scale=0.4]{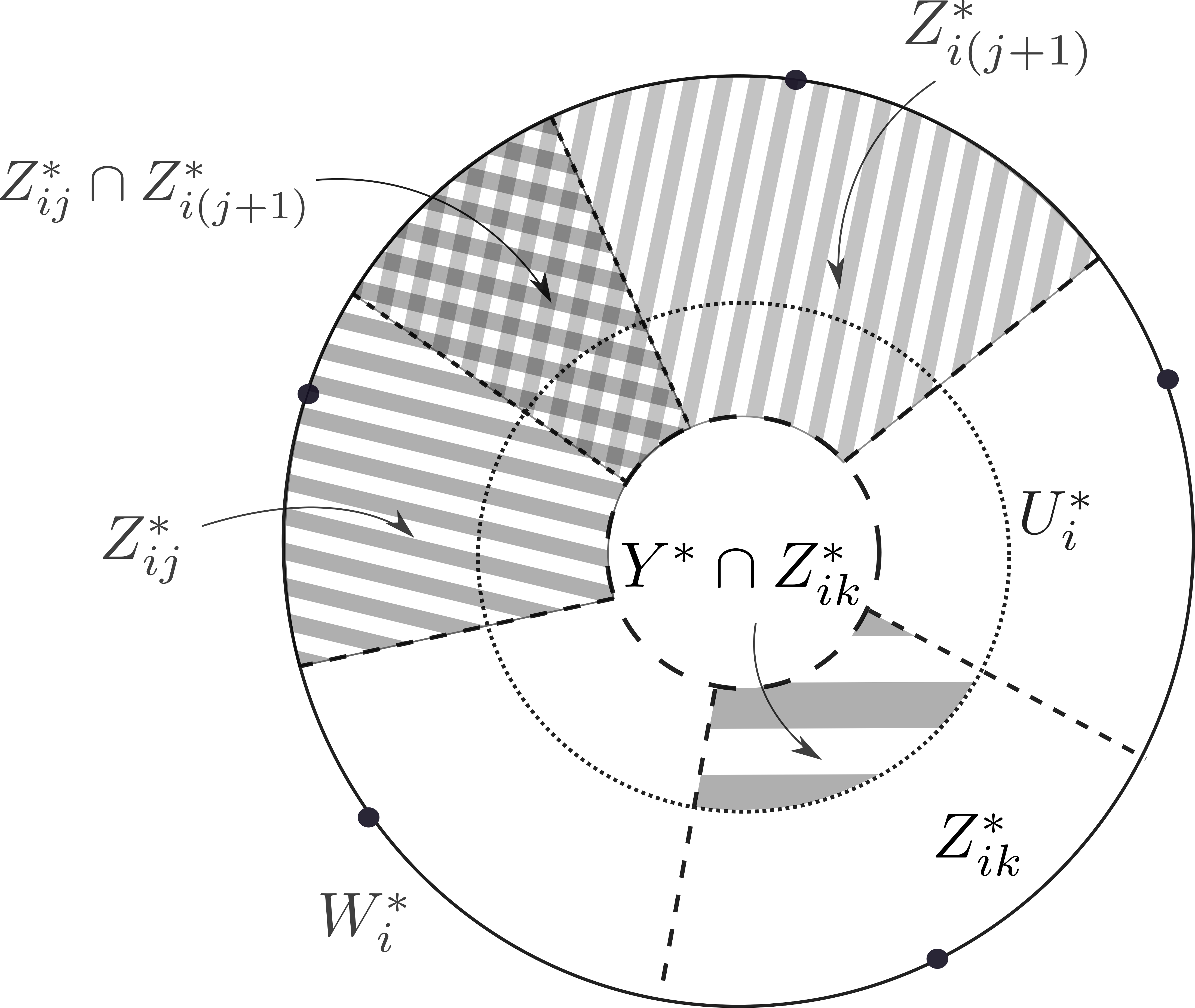}
\caption{Open cover of $X^\ast$.}
\label{FIG_OPEN_COVER}
\end{figure} 


Now we  construct smooth $4$-orbifolds equipped with effective $\soso$-actions over each of the pieces of the open cover $\{Y^\ast\}\cup(\cup_{ij}\{Z_{ij}^\ast\}$, and show that on the non-trivial intersections they are compatible in the sense of Proposition~\ref{T:Haefliger-Salem-orbifolds}.

\subsection*{Orbifold corresponding to $\mathbf{Y^*}$.} We point out that the number of boundary components in $X^\ast$ containing only $RF$-orbits is given by $q-s$ in the notation above. With this observation we point out that  by \cite{OrlRayI,OrlRayII}, the set of invariants 
\[
\big\{(b_1,b_2);\varepsilon;g; \big\{ \left\langle p_i,q_i\right\rangle \big\}_{i=1}^{s}; \big\{ \big((a_{l},b_{l}),\pm 1\big)  \big\}_{l=1}^{q-s} ; \big\{ (\alpha_j;\gamma_{j,1},\gamma_{j,2}) \big\}_{j=1}^k \big\}
\]
determines a unique (up to equivariant homeomorphism) closed, smooth $4$-manifold $\hat{Y}$ with an effective $\soso$-action having these invariants. Note that $\hat{Y}^\ast$ has exactly $q\geq 0$, boundary components: $s$ boundary components containing only circular isotropy; and $q-s$ boundary components containing at least one regular fixed point. Recall that $t$ is the number of boundary components of $X^\ast$ containing fixed points. The number of boundary components containing $SF$-orbits is given by  $t+s-q$. Let $\hat{Y}_1^\ast$  be  the subset of $\hat{Y}^\ast$ obtained by taking out sufficiently small, disjoint, closed  neighborhoods of $t+s-q$ (arbitrary) points with trivial isotropy. Then $\hat{Y}_1^\ast$ is isomorphic to $Y^\ast$. Set $Y$ to be  the inverse image of $\hat{Y}_1^\ast$ under the projection map $\hat{Y}\to \hat{Y}^\ast$, associated to the $\soso$-action on $\hat{Y}$. Notice that $Y$ is an  open submanifold of $\hat{Y}$ which is invariant under the  action of $\soso$.

\subsection*{Orbifolds corresponding to $\mathbf{Z_{ij}^*}$.} Let $x_{ij}^\ast\in Z_{ij}^\ast$ be the fixed point in this neighborhood  with  isotropies $G(a_{ij},b_{ij})$ and $G(a_{i(j+1)},b_{i(j+1)})$, with $a_{ij}b_{i(j+1)}-a_{i(j+1)}b_{ij}=f_{ij}$. 
From Section~\ref{Sub: Local structure of X^* around points in S_X^0} and Section~\ref{Sub: Regular Fixed points}, we know that the open cone over a lens space has an $\soso$-action such that the orbit space is isomorphic to $Z_{ij}^\ast$. We  point out that  the cone over a lens space admits an orbifold structure.  We denote such a cone by $Z_{ij}$.


\subsection*{Compatibility of the open cover}
We now show that over the overlaps in $X^\ast$, the smooth $4$-orbifolds we have produced are equivariantly diffeomorphic. 
That is, we need to examine the intersection of $Z_{ik}^\ast$ with $ Y^\ast$, and the intersection of $Z_{ij}^\ast$ with $Z_{i(j+1)}^\ast$, for $j =1,\ldots, r_i$ with the identification $r_i+1 = 1$.

In the first case, the intersection $Z_{ik}^\ast \cap Y^\ast$ is the product of two open intervals $I_1\times I_2$ (see Figure~\ref{FIG_OPEN_COVER}). Set $A = \pi^{-1}_{Y}(I_1\times I_2)\subset Y$. Observe that $A$ is contained in a smooth manifold, and the action of $\soso$ on $A$ is free. Since $A/\soso = I_1\times I_2$ is contractible, we conclude that $A$ is equivariantly diffeomorphic to $\soso\times I_1\times I_2$ via a diffeomorphism $\psi_A\colon \soso\times I_1\times I_2\to A$. Analogously we point out, that  $B = \pi^{-1}_{Z_{ik}}(I_1\times I_2)\subset Z_{ik}$ is a manifold,  and the action of $\soso$ on $B$ is free. Since $B/\soso = I_1\times I_2$ is contractible, we conclude that there exists an equivariant diffeomorphism $\psi_{B}\colon B \to \soso  \times I_1\times I_2$. Therefore, there exists an equivariant diffeomorphism $\Psi\colon B\to A$ defined as $\Psi = \psi^{-1}_A\circ\psi_B$.

We now focus ourselves on $Z_{ij}^\ast\cap Z_{i(j+1)}^\ast$. We can consider the closure of $Z_{ij}^\ast\cap Z_{i(j+1)}^\ast$ in $X^\ast$, and since $Z_{ij}$ and $Z_{i(j+1)}$ are open cones over lens spaces, we can  then assume $Z_{ij}$ and $Z_{i(j+1)}$ to be the closed cones.  The intersection of the closures is homeomorphic to $[0,1]\times[0,1]$. Moreover, the edge $[0,1]\times\{0\}$ consists of orbits with isotropy $G(a_{i(j+1)},b_{i(j+1)})$, for both actions of $\soso$ on $Z_{ij}$ and $Z_{i(j+1)}$ (see Figure~\ref{FIG_OPEN_COVER}). Then by the proof of Lemma~\ref{L:Cross-section circle orbits}, the subset $A =\pi^{-1}_{Z_{ij}}([0,1]\times[0,1])\subset Z_{ij}$ is equivariantly diffeomorphic to $[0,1]\times K(S^1)\times S^1$, where $K(S^1)$ is the closed cone over $S^1$. This same statement holds for $B =\pi^{-1}_{Z_{i(j+1)}}([0,1]\times[0,1])\subset Z_{i(j+1)}$: it is equivariantly diffeomorphic to $[0,1]\times K(S^1)\times S^1$. Thus we conclude that $A$ is equivariantly diffeomorphic to $B$.

Thus the conditions of compatibility in Proposition~\ref{T:Haefliger-Salem-orbifolds} are fulfilled and this procedure gives rise to a cohomology class in $H^3(X^\ast, \Z^2)$. Since $X^\ast$ is a $2$-dimensional manifold, the group  $H^3(X^\ast, \Z^2)$ vanishes, and the conclusion of  Proposition~\ref{T:Haefliger-Salem-orbifolds} gives us  a compact, smooth $4$-orbifold $X$ (possibly with boundary) carrying an effective $\soso$-action which has exactly the collection
\[
\big\{(b_1,b_2);\varepsilon;g; \big\{ \left\langle p_i,q_i\right\rangle \big\}_{i=1}^s; \big\{ \big( (a_{\ell},b_{\ell}),f_{\ell}\big)\big\}_{\ell=1}^{t} ; \big\{ (\alpha_j;\gamma_{j,1},\gamma_{j,2}) \big\}_{j=1}^k \big\}
\]
as its set of invariants. That is, $X/\soso$ is isomorphic to $X^\ast$. 
\end{proof}

Thus for a weighted surface $X^\ast$, we have constructed a space $X$ (the underlying space of a smooth $4$-orbifold) which admits a $\soso$-action, and realizes $X^\ast$ as its orbit space. Now, we point out that this space $X$ admits an Alexandrov metric.

\begin{thm}\label{T: Each legally weighted 2-space represents an Alex space}
For each legally weighted surface $X^\ast$, there exists a closed, orientable Alexandrov $4$-space admitting an effective and isometric $\soso$-action realizing $X^\ast$ as its orbit space.
\end{thm}

\begin{proof}
Let $X$ be the underlying space of the orientable orbifold $\mathcal{O}_X$ obtained by Theorem~\ref{T: existence of orbifold}. As in the case of Riemannian manifolds, for the $\soso$-action on $\mathcal{O}_X$ we can construct an invariant orbifold metric via \cite[Theorem~3.65]{AleBet}. This induces the desired invariant Alexandrov metric.
\end{proof}

Combining Theorems~\ref{T:Equivalent iff orbit spaces are isomorphic}, \ref{T: existence of orbifold},  and \ref{T: Each legally weighted 2-space represents an Alex space}  we conclude the following

\begin{cor}
Let $X$ be a closed orientable $4$-dimensional Alexandrov space with a $\soso$-action by isometries. Then $X$ is equivariantly homeomorphic to an orbifold.
\end{cor}

\section{Basic topological recognition}
\label{S:BASIC_TOPOLOGICAL_RECOGNITION}

In this section we prove a basic topological recognition result: we show that given a closed, orientable Alexandrov $4$-space with an effective, isometric $\soso$-action, it is always possible to decompose it as an ``equivariant connected sum''  of 
a closed smooth $4$-manifold with an effective $T^2$-action, with a finite number of  ``simple'' closed, orientable Alexandrov $4$-spaces with standard $\soso$-actions; the connected sum is done along  the boundaries of tubular neighborhoods of some circle orbits. 
To make this statement precise we define these simple closed Alexandrov $4$-spaces with an isometric $\soso$-action.

\begin{definition} Let $X$ be a closed, orientable Alexandrov $4$-space with an effective, isometric $\soso$-action. We say that $X$ is \textit{simple} if $X^\ast$ is homeomorphic to a $2$-disk and $E=\emptyset$, $C=\emptyset$ and $RF\cup SF\neq \emptyset$. 
\end{definition}

We observe that the set of invariants \eqref{EQ:INVARIANTS} associated to a simple Alexandrov $4$-space is of the form 
\[
\left\{(0,0);\varepsilon;0; - ; \big\{\big((a_{1,1},b_{1,1}),f_{1,1},(a_{1,2},b_{1,2}),f_{1,2},\ldots, (a_{1,r_{1}},b_{1,r_{1}}), f_{1,r_{1}} \big)\big \}; -  \right\}.
\]

Before proving the decomposition result of closed orientable $4$-dimensional Alexandrov spaces with a $\soso$-action by isometries, we present some  examples of simple Alexandrov spaces. At the present moment we are not aware of a complete list of 
simple spaces.

\begin{example}
Let $\soso$ act on the lens space $L(r, s)$ of cohomogeneity one with the group diagram $(\soso, 1, G(p,q), G(m,n))$, where $pn-mq=r$.  Then $\susp(L(r, s))$ with the suspension action of $\soso$  is a simple Alexandrov space.  This action has two singular fixed points and its set of invariants  is as follows:

\[
\left\{(0,0);\varepsilon;0; - ; \big\{\big( (p,q),r , (m, n), -r\big)\big\}; -  \right\}.
\]
\end{example}

\begin{example}\label{EX: weighted projective space}
This example deals with a $\soso$-action on a weighted projective space, which is defined as follows: Consider the unit sphere
\[
\Ss^5 =\{(z_1, z_2, z_3)\mid \sum_{i=1}^3|z_i|^2=1\}\subseteq \mathbb{C}^3,
\] 
and define an action of the circle $\Circ$ on $\Ss^5$ as follows:

\[
e^{i\varphi}.(z_1, z_2, z_3)=(e^{ir_1\varphi}z_1, e^{ir_2\varphi}z_2, e^{ir_3\varphi}z_3).
\]
Here the $r_i$'s, are coprime integers. The quotient space  of the sphere by this action 
\[
W\mathbb{P}(r_1, r_2, r_3)=\Ss^5/\Circ
\]
is called a \emph{weighted projective space}. Note that $W\mathbb{P}(r_1, r_2, r_3)$ is a compact Riemannian orbifold and in particular an Alexandrov space. We now define an effective $\soso$-action on $W\mathbb{P}(r_1, r_2, r_3)$.  We assume that each $r_i$ is positive, and we choose three pairs of coprime integers  $(m_1, n_1)$, $(m_2, n_2)$, and $(m_3, n_3)$, satisfying  
\begin{align*}
m_1n_{2}-m_{2}n_1 &= r_{2},\\
m_2n_{3}-m_{3}n_2 &= -r_{3},\\
m_3n_{1}-m_{1}n_3 &= r_{1}.
\end{align*}

We define an $\soso$-action on $\Ss^5$ as 
\[
(\varphi, \theta).(z_1, z_2, z_3)=(z_1, e^{i\frac{m_3\varphi+n_3\theta}{r_1}}z_2, e^{i\frac{m_1\varphi+n_1\theta}{r_1}}z_3).
\]
Since the actions of the circle  $\Circ$ commuts with the action of $\soso$ on $\Ss^5$,  the action of $\soso$ on $\Ss^5$ descends to an action on $W\mathbb{P}(r_1, r_2, r_3)$. This action has three fixed points and the set of invariants of the action is:

\[
\left\{(0,0);\varepsilon;0; - ; \big\{\big((m_1,n_1),r_2 , (m_2, n_2), -r_3, (m_3, n_3), r_1\big)\big\}; -  \right\}.
\]
\end{example}

\subsection*{$C$-Equivariant connected sums}

We describe a similar construction to that of the usual \textit{equivariant connected sums} which are performed at $RF$-orbits (see \cite{OrlRayI, OrlRayII}). The core difference is that it is performed at $C$-orbits.

Let $X_1$ and $X_2$ be two closed, orientable Alexandrov $4$-spaces equipped with effective and isometric $\soso$-actions, and let $\pi_i: X_i \to X_i^\ast$ denote the canonical projection maps for $i=1,2$.  We let $C_i$, $RF_i$, $SF_i$ and $E_i$ denote the subsets of $X_i$ consisting of points on circular orbits, topologically regular fixed points, topologically singular fixed points and exceptional orbits respectively for $i=1,2$. We assume that  $C_i\neq \emptyset$ for $i=1,2$. Furthermore, we assume that there exist circular orbits $x_i^\ast\in X_i^\ast$ having the same isotropy, that is, there exist $x_{i}\in X_{i}$ with $\pi_{i}(G(x_{i}))=x_{i}^{\ast}$ and $G_{x_{i}}=G(m,n)$ for some $m,n$, with  $\mathrm{gcd}(m, n)=1$, for $i=1,2$. 

Now we consider invariant closed tubular neighborhoods $\overline{B_i}$ of each orbit $\pi_{i}^{-1}(x_i)$,  sufficiently small so that $\overline{B_i}\cap SF_i= \overline{B_i}\cap RF_i=\overline{B_i}\cap E_i=\emptyset$ for $i=1,2$.
Since the actions are by  isometries, $\partial \overline{B_i}$ is an invariant subset of $X_i$.   It follows from the classification of cohomogeneity-one $3$-manifolds \cite{Neu} (see Subsection \ref{SS:Cohomogeneity-one 3-manifolds}) that the restriction of the $\soso$-action on $X_i$ to $\partial \overline{B_i}$ is equivalent to the cohomogeneity-one action with group diagram $(\soso, G(m,n), G(m,n), \{1\})$. Note that this in particular shows that $\partial \overline{B_i}$ is equivariantly homeomorphic to $S^2\times S^1$ for $i=1,2$ (see Table~\ref{TBL:cohom1-3mflds}). Hence, there exists an equivariant homeomorphism $\Psi: \partial \overline{B_1}\to \partial\overline{B_2}$. 

In a similar fashion to that of the $RF$-equivariant connected sums we consider a space $Z$ obtained by gluing $X_1\setminus B_1$ with $X_2\setminus B_2$ along their homeomorphic $S^2\times S^1$ boundaries via $\Psi$.

\begin{definition} We say that the space $Z$ is the \textit{$C$-equivariant connected sum of $X_1$ and $X_2$ with respect to the circular orbits $x_1^\ast$ and $x_2^\ast$} and denote it by $X_1\#_C X_2$.
\end{definition}

\begin{rem}\label{R:Gluing_Map_in_the_orbit_space}
Since $\partial\overline{B_{i}}$ is  $\soso$-invariant and $\Psi$ is an equivariant homeomorphism, one can define a homeomorphism $\Psi^{\ast}\colon (\partial\overline{B_{1}})^{\ast}\to (\partial\overline{B_{2}})^{\ast}$ such that the following diagram commutes:
\begin{figure}[H]
\centering
\centering
\begin{tikzpicture}[]
\draw [->] (-1.6, 0) -- (-.6,0) ;
\node at (-1.7,0) [left ]{$\ \ \ \ \partial\overline{B_{1}}$};
\node at (-.2,0)  {$\ \ \ \ \partial\overline{B_{2}}$};
\draw [->] (-1.6,-2) -- (-0.6,-2) ;
\node at (-1.7,-2) [left ]{$\ \ \ \  (\partial\overline{B_{1}})^{\ast}$};
\node at (1.1,-2) [left ]{$\ \ \ \  (\partial\overline{B_{2}})^{\ast}.$};
\draw [->] (-2, -.3) -- (-2,-1.7) ;
\draw [->] (-.2, -.3) -- (-0.2,-1.7) ;
\node at (-1.4, 0) [above ]{$\ \ \ \  \Psi$};
\node at (-1.4,-2) [below ]{$\ \ \ \  \Psi^{\ast}$};
\node at (-0.6, -1) [right ]{$\ \ \ \  \pi_{2}$};
\node at (-2.1,-1) [left]{$\ \ \ \ \pi_{1}$};
\end{tikzpicture}
\end{figure}

\end{rem}

Recall from Subsection~\ref{SS:One dimensional orbits} and the Slice Theorem that $\overline {B_{i}}$ is equivariantly homeomorphic to $(\soso\times C(\Ss^{2}))/ G(m,n)$,  where $C(\Ss^{2})$ is a closed cone over $\Ss^{2}$, and its  boundary  is equivariantly  homeomorphic  to $(\soso\times \Ss^{2}\times 1)/G(m,n)$.  Note that   ${B_i}$ is mapped to   the shaded subset of $X_{i}^{\ast}$ as in Figure~\ref{FIG:Circular_Orbit_C_Connected} and its boundary is mapped to the arc $\arc{abc}$.   

\begin{figure}[ht]
\begin{center}
\begin{tikzpicture}[rotate=00,scale = 0.6]
\draw[dashed] (0,3) arc (0:90:4cm and 1cm);
\draw [dashed] (-4,4) arc (90:180:4cm and 1cm);
\node   at (-4,2.1) [above] {$\ \ \ \  (m, n)$}; 
\draw (-1.6,-3) arc (35:145:2.5cm and 4cm); 
     \begin{scope}  
  \clip  (-5,2.05) arc (180:365:1.5cm and 1.4cm); 
  \path[fill=white!25!gray](-5,2.05) arc (180:365:1.5cm and 1.4cm); 
   \end{scope}
     \begin{scope} 
      \clip (0,3) arc (0:-90:4cm and 1cm);
  \path[fill=white](-5,2.05) arc (180:365:1.5cm and 1.4cm); 
    \end{scope}
     \begin{scope} 
  \clip(-8,3) arc (180:283:4cm and 1cm);
  \path[fill=white](-5,2.05) arc (180:365:1.5cm and 1.4cm); 
    \end{scope}
\draw  (0,3) arc (0:-90:4cm and 1cm);
\draw (-8,3) arc (180:270:4cm and 1cm);
  \draw[thick](-5,2.05) arc (180:364:1.5cm and 1.4cm); 
\draw (-8.5,-2.86) arc (-26:13:6cm and 9cm); 
\draw(1.30,-2.4) arc (-68:-1:-2cm and 6cm); 
\draw  [dashed](-6,-0.5) arc (0:50:3cm and 1.5cm); 
\draw  (-6,-0.5) arc (-50:0:-3cm and 1.5cm) ; 
\draw [dashed](-2,-0.3) arc (0:50:-3cm and 1.5cm); 
\draw (-2,-0.3) arc (-50:0:3cm and 1.5cm) ; 
\node   at (-5.7,2.05) [below] {$\ \ \ \  a$};
\node   at (-4.2,0.8) [below] {$\ \ \ \  b$};
\node   at (-2.2,2.05) [below] {$\ \ \ \  c$};
\draw [color = black, fill=black] (-3.6,2) circle (1pt);
\node   at (-4,1.2) [above] {$\ \ \ \  x_{i}$}; 
\end{tikzpicture}
\caption{The image of $\overline{B_{i}}$ in $X_{i}^{\ast}$.}
\label{FIG:Circular_Orbit_C_Connected}
\end{center}
\end{figure}
Hence, the image of $X_{i}\setminus B_{i}$ is merely  $X_{i}^{\ast}$ with the shaded area removed (see Figure~\ref{FIG:Gluing}).  It is clear from the gluing procedure that $Z^{\ast}$ is just the gluing of $(X_1\setminus B_1)^{\ast}$ and $(X_2\setminus B_2)^{\ast}$ along $(\partial B_{1})^{\ast}$ and $(\partial B_{2})^{\ast}$ via the homeomorphism $\Psi^{\ast}$ as illustrated in Figure~\ref{FIG:Gluing}, for $X_1$ and $X_2$ simply connected.

\begin{figure}[ht]
\begin{center}

\includegraphics[scale=0.5]{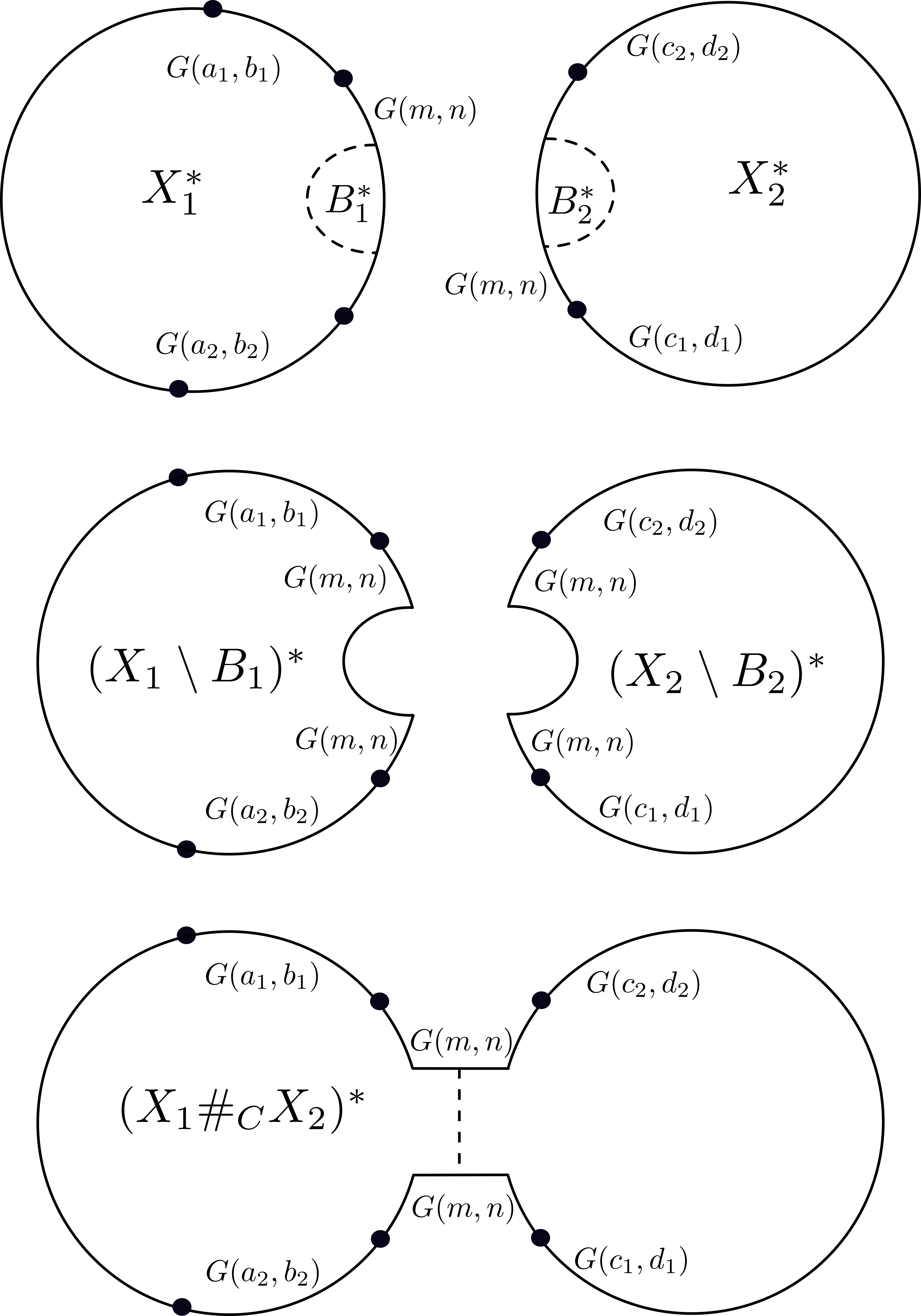}
\end{center}
\caption{Gluing $(X_1\setminus B_1)^{\ast}$ and $(X_2\setminus B_2)^{\ast}$ along $(\partial\overline{B_{i}})^{\ast}$ via $\Psi^{\ast}$.}
\label{FIG:Gluing}
\end{figure}

Therefore, we have shown the following lemma. 

\begin{lem}\label{L:Orbit_Space_of_C-Connected_Sum}
Let $X_{i}$, $i=1, 2$ and $X_1\#_C X_2$ be as above. Then we have 
$$(X_1\#_CX_2)^\ast=(X_1\setminus B_1)^{\ast}\#_C(X_2\setminus B_2)^{\ast},$$
where $\#_{C}$ on the right hand side means the gluing along $(\partial\overline{B_{i}})^{\ast}$ via the map $\Psi^{\ast}$ defined in Remark~\ref{R:Gluing_Map_in_the_orbit_space}. In particular,  if the boundary component of $X_{1}^{\ast}$ containing $x_{1}^{\ast}$ has only $C$-orbits with isotropy $G(m, n)$, then the boundary component of $Z$ obtained as a result of $C$-equivariant connected sum and the boundary component of $X_{2}$ containing $x_{2}$ have the same weights. 
\end{lem}

Let us make some remarks. The equivariant homeomorphism type of $X_1\#_C X_2$ depends on the choices of the circular orbits used in the construction. For the sake of simplicity, the notation we use does not explicitly display this fact. The space $X_1\#_CX_2$ is naturally equipped with an effective $\soso$-action. Moreover, as the orbit space of such action is legally weighted, it follows from Theorem \ref{T: existence of orbifold} that $X_1\#_CX_2$ is equivariantly homeomorphic to the closed, orientable Alexandrov $4$-space having an effective and isometric $\soso$-action (uniquely determined up to equivariant homeomorphism) having  $(X_1\#_CX_2)^\ast$ as orbit space. 

With these definitions in hand we state the main result of the section. 

\begin{thm}
Let $X$ be a closed, orientable Alexandrov $4$-space with an effective, isometric $\soso$-action. Then $X$ is equivariantly homeomorphic to $C$-equivariant connected sums of a single closed orientable $4$-manifold with an effective $\soso$-action and a collection of simple Alexandrov $4$-spaces. 
\end{thm}

\begin{proof}
We begin by noting that if $SF=\emptyset$ then the result is immediate as $X$ is equivariantly homeomorphic to a $4$-manifold and there is no need to take $C$-equivariant connected sums. Therefore, we assume that $SF\neq \emptyset$. 

Let $X$ be determined by the following set of invariants
\[ 
\big\{(b_1,b_2);\varepsilon;g; \big\{ \left\langle p_i,q_i\right\rangle \big\}_{i=1}^s; \big\{ \big( (a_{\ell},b_{\ell}),f_{\ell}\big) \big\}_{\ell=1}^{t} ; \big\{ (\alpha_j;\gamma_{j,1},\gamma_{j,2}) \big\}_{j=1}^k \big\}.
\]
By our assumption on $SF$ we have that $t>0$. Then we can assume without loss of generality, by relabeling the boundary components if necessary, that there exists a maximal index $0\leq q< t$ such that for each $1\leq \ell \leq q$  we have $f_{\ell,w}=\pm 1$  for all $1\leq \ell \leq q$; i.e.  up to the index $q$ all the boundary components containing fixed points do not contain $SF$-points. 


We now consider  a collection of simple Alexandrov $4$-spaces $X_j$ with $j=1,\ldots, t-q$, given by
\[
\left\{(0,0);\varepsilon;0; - ; \big\{ \big( (a^j_{1,1},b^j_{1,1}),f^j_{1,1}, (a^j_{1,2},b^j_{1,2}), f^j_{1,2},\ldots, (a^j_{1,r_{j+q}}, b^j_{1,r_{j+q}}),f^j_{1,r_{j+q}} \big) \big\}; -  \right\},
\]
satisfying that $(a^j_{1,w},b^j_{1,w})=(a_{j+q,w}, b_{j+q,w})$ for all $w=1,\ldots, r_{j+q}$. 
We also consider a closed, orientable $4$-manifold $M$ with an effective $\soso$-action determined by the invariants
\[
 \big\{(b_1,b_2);\varepsilon;g; \big\{ \left\langle p_i,q_i\right\rangle \big\}_{i=1}^{\overline{s}}; \big\{ \big( (\overline{a}_{\ell},\overline{b}_{\ell}),\overline{f}_{\ell}\big) \big\}_{\ell=1}^{\overline{t}} ; \big\{ (\alpha_j;\gamma_{j,1},\gamma_{j,2}) \big\}_{j=1}^k \big\},
\]
where we recall that $\big((\overline{a}_{\ell},\overline{b}_{\ell}),\overline{f}_{\ell}\big)$ denotes $\big((\overline{a}_{\ell,1},\overline{b}_{\ell,1}),\overline{f}_{\ell,1},\ldots,(\overline{a}_{\ell,\overline{r}_{\ell}},\overline{b}_{\ell,\overline{r}_{\ell}}),\overline{f}_{\ell,\overline{r}_{\ell}} \big)$, satisfying the following conditions:

\begin{itemize}
\item[(i)] $\overline{s}= s+t-q$ and $\overline{t}= q$,
\item[(ii)] $(\overline{a}_{\ell,w},\overline{b}_{\ell,w})= (a_{\ell,w},b_{\ell,w})$ (and therefore $\overline{f}_{\ell,w}=f_{\ell,w}=\pm 1$) for all $w=1,\ldots,\bar{r}_{\ell}$, and $\ell=1,\ldots, \overline{t}$,
\item[(iii)] $\left\langle p_i,q_i \right\rangle = \left\langle a_{i+q,1}, b_{i+q,1} \right\rangle$ for all $i=1,\ldots, t-q$. 
\end{itemize}


By the third point in the conditions defining $M$, for each $j=1,\ldots, t-q$ there exist a circular orbit $x^\ast\in M^\ast$ contained in the $j$-th connected component of $\partial M^\ast$ containing only circular orbits and a circular orbit $x_j^\ast\in X_j^\ast$ having the same weight, namely $(a_{j+q,1}, b_{j+q,1})$. Therefore the $C$-equivariant connected sum of $M$ and $X_j$ can be performed with respect to $x^\ast$ and $x_j^\ast$ for all $j=1,\ldots, t-q$ yielding a space $M\#_CX_1\#_C\ldots \#_C X_{t-q}$ with an effective $\soso$-action whose orbit space is isomorphic to $X^\ast$ by Lemma~\ref{L:Orbit_Space_of_C-Connected_Sum}. Hence by Theorem \ref{T:Equivalent iff orbit spaces are isomorphic} we have that $X$ is equivariantly homeomorphic to $M\#_CX_1\#_C\ldots \#_C X_{t-q}$. 
\end{proof}

\begin{rem} To achieve a full homeomorphism classification of closed, orientable Alexandrov $4$-spaces with effective and isometric $\soso$-actions it is sufficient to have a classification of the homeomorphism type of simple Alexandrov $4$-spaces and combine this with the previous homeomorphism classification of Orlik-Raymond \cite{OrlRayI, OrlRayII} and Pao \cite{Pao2,Pao1}. 
\end{rem}

\bibliographystyle{siam}
\nocite{*}
\bibliography{Bibliography}

\begin{bibdiv}
\begin{biblist}

\bib{AleBet}{book}{
      author={Alexandrino, Marcos~M.},
      author={Bettiol, Renato~G.},
       title={Lie groups and geometric aspects of isometric actions},
   publisher={Springer, Cham},
        date={2015},
        ISBN={978-3-319-16612-4; 978-3-319-16613-1},
         url={https://doi.org/10.1007/978-3-319-16613-1},
      review={\MR{3362465}},
}

\bib{AZ}{article}{
      author={Amann, Manuel},
      author={Zarei, Masoumeh},
       title={On the equivariant cohomology of cohomogeneity one {A}lexandrov
  spaces},
        date={2019},
     journal={\href{https://arxiv.org/pdf/1910.06309.pdf}{arXiv:1910.06309v2}
  [math.DG]},
}

\bib{Be}{article}{
      author={Berestovski\u{\i}, V.~N.},
       title={Homogeneous manifolds with an intrinsic metric. {II}},
        date={1989},
     journal={Sibirsk. Mat. Zh.},
      volume={30},
      number={2},
       pages={14\ndash 28, 225},
         url={https://doi.org/10.1007/BF00971372},
      review={\MR{997464}},
}

\bib{Bre}{book}{
      author={Bredon, Glen~E.},
       title={Introduction to compact transformation groups},
   publisher={Academic Press, New York-London},
        date={1972},
        note={Pure and Applied Mathematics, Vol. 46},
      review={\MR{0413144}},
}

\bib{Bro}{article}{
      author={Brody, E.~J.},
       title={The topological classification of the lens spaces},
        date={1960},
     journal={Ann. of Math. (2)},
      volume={71},
       pages={163\ndash 184},
         url={https://doi.org/10.2307/1969884},
      review={\MR{116336}},
}

\bib{BurBurIva}{book}{
      author={Burago, Dmitri},
      author={Burago, Yuri},
      author={Ivanov, Sergei},
       title={A course in metric geometry},
      series={Graduate Studies in Mathematics},
   publisher={American Mathematical Society, Providence, RI},
        date={2001},
      volume={33},
        ISBN={0-8218-2129-6},
         url={https://doi.org/10.1090/gsm/033},
      review={\MR{1835418}},
}

\bib{BurGroPer}{article}{
      author={Burago, Yu.},
      author={Gromov, M.},
      author={Perel\cprime~man, G.},
       title={A. {D}. {A}leksandrov spaces with curvatures bounded below},
        date={1992},
     journal={Uspekhi Mat. Nauk},
      volume={47},
      number={2(284)},
       pages={3\ndash 51, 222},
         url={https://doi.org/10.1070/RM1992v047n02ABEH000877},
      review={\MR{1185284}},
}

\bib{ChoLau}{book}{
      author={Chossat, Pascal},
      author={Lauterbach, Reiner},
       title={Methods in equivariant bifurcations and dynamical systems},
      series={Advanced Series in Nonlinear Dynamics},
   publisher={World Scientific Publishing Co., Inc., River Edge, NJ},
        date={2000},
      volume={15},
         url={https://doi.org/10.1142/4062},
      review={\MR{1831950}},
}

\bib{Zie}{article}{
      author={Ci\v{s}ang, H.},
       title={Torus bundles over surfaces},
        date={1969},
     journal={Mat. Zametki},
      volume={5},
       pages={569\ndash 576},
      review={\MR{248835}},
}

\bib{ConRay}{inproceedings}{
      author={Conner, P.~E.},
      author={Raymond, Frank},
       title={Holomorphic {S}eifert fiberings},
        date={1972},
   booktitle={Proceedings of the {S}econd {C}onference on {C}ompact
  {T}ransformation {G}roups ({U}niv. {M}assachusetts, {A}mherst, {M}ass.,
  1971), {P}art {II}},
       pages={124\ndash 204. Lecture Notes in Math., Vol. 299},
      review={\MR{0590802}},
}

\bib{DanWaer}{article}{
      author={Dantzig, D.~van},
      author={Waerden, B. L. van~der},
       title={\"{U}ber metrisch homogene r\"{a}ume},
        date={1928},
     journal={Abh. Math. Sem. Univ. Hamburg},
      volume={6},
      number={1},
       pages={367\ndash 376},
         url={https://doi.org/10.1007/BF02940622},
      review={\MR{3069509}},
}

\bib{FukYam}{article}{
      author={Fukaya, Kenji},
      author={Yamaguchi, Takao},
       title={Isometry groups of singular spaces},
        date={1994},
     journal={Math. Z.},
      volume={216},
      number={1},
       pages={31\ndash 44},
         url={https://doi.org/10.1007/BF02572307},
      review={\MR{1273464}},
}

\bib{Gal}{article}{
      author={Galaz-Garcia, Fernando},
       title={Simply connected {A}lexandrov {4}-manifolds with positive or
  nonnegative curvature and torus actions},
        date={2012},
     journal={\href{https://arxiv.org/pdf/1208.3041.pdf}{arXiv:1208.3041}
  [math.DG]},
}

\bib{Gal2016}{article}{
      author={Galaz-Garc\'{\i}a, Fernando},
       title={A glance at three-dimensional {A}lexandrov spaces},
        date={2016},
     journal={Front. Math. China},
      volume={11},
      number={5},
       pages={1189\ndash 1206},
         url={https://doi.org/10.1007/s11464-016-0582-3},
      review={\MR{3547925}},
}

\bib{GalGui13}{article}{
      author={Galaz-Garcia, Fernando},
      author={Guijarro, Luis},
       title={Isometry groups of {A}lexandrov spaces},
        date={2013},
     journal={Bull. Lond. Math. Soc.},
      volume={45},
      number={3},
       pages={567\ndash 579},
         url={https://doi.org/10.1112/blms/bds101},
      review={\MR{3065026}},
}

\bib{GalGui}{article}{
      author={Galaz-Garcia, Fernando},
      author={Guijarro, Luis},
       title={On three-dimensional {A}lexandrov spaces},
        date={2015},
     journal={Int. Math. Res. Not. IMRN},
      number={14},
       pages={5560\ndash 5576},
         url={https://doi.org/10.1093/imrn/rnu101},
      review={\MR{3384449}},
}

\bib{GaKe}{article}{
      author={Galaz-Garcia, Fernando},
      author={Kerin, Martin},
       title={Cohomogeneity-two torus actions on non-negatively curved
  manifolds of low dimension},
        date={2014},
     journal={Math. Z.},
      volume={276},
      number={1-2},
       pages={133\ndash 152},
         url={https://doi.org/10.1007/s00209-013-1190-5},
      review={\MR{3150196}},
}

\bib{GaKeRaWe}{article}{
      author={Galaz-Garc\'{\i}a, Fernando},
      author={Kerin, Martin},
      author={Radeschi, Marco},
      author={Wiemeler, Michael},
       title={Torus orbifolds, slice-maximal torus actions, and rational
  ellipticity},
        date={2018},
     journal={Int. Math. Res. Not. IMRN},
      number={18},
       pages={5786\ndash 5822},
         url={https://doi.org/10.1093/imrn/rnx064},
      review={\MR{3862119}},
}

\bib{GalNun}{article}{
      author={Galaz-Garc\'{\i}a, Fernando},
      author={N\'{u}\~{n}ez Zimbr\'{o}n, Jes\'{u}s},
       title={Three-dimensional {A}lexandrov spaces with local isometric circle
  actions},
        date={2020},
     journal={Kyoto J. Math.},
      volume={60},
      number={3},
       pages={801\ndash 823},
         url={https://doi.org/10.1215/21562261-2019-0047},
      review={\MR{4134350}},
}

\bib{GalSea}{article}{
      author={Galaz-Garcia, Fernando},
      author={Searle, Catherine},
       title={Cohomogeneity one {A}lexandrov spaces},
        date={2011},
     journal={Transform. Groups},
      volume={16},
      number={1},
       pages={91\ndash 107},
         url={https://doi.org/10.1007/s00031-011-9122-0},
      review={\MR{2785496}},
}

\bib{GalSea11}{article}{
      author={Galaz-Garcia, Fernando},
      author={Searle, Catherine},
       title={Low-dimensional manifolds with non-negative curvature and maximal
  symmetry rank},
        date={2011},
     journal={Proc. Amer. Math. Soc.},
      volume={139},
      number={7},
       pages={2559\ndash 2564},
         url={https://doi.org/10.1090/S0002-9939-2010-10655-X},
      review={\MR{2784821}},
}

\bib{GalZar}{article}{
      author={Galaz-Garc\'{\i}a, Fernando},
      author={Zarei, Masoumeh},
       title={Cohomogeneity one topological manifolds revisited},
        date={2018},
     journal={Math. Z.},
      volume={288},
      number={3-4},
       pages={829\ndash 853},
         url={https://doi.org/10.1007/s00209-017-1915-y},
      review={\MR{3778980}},
}

\bib{GalZar2}{article}{
      author={Galaz-Garc\'{\i}a, Fernando},
      author={Zarei, Masoumeh},
       title={Cohomogeneity one {A}lexandrov spaces in low dimensions},
        date={2020},
     journal={Ann. Global Anal. Geom.},
      volume={58},
      number={2},
       pages={109\ndash 146},
         url={https://doi.org/10.1007/s10455-020-09716-7},
      review={\MR{4129520}},
}

\bib{Ghy}{article}{
      author={Ghys, \'{E}tienne},
       title={Groups acting on the circle},
        date={2001},
     journal={Enseign. Math. (2)},
      volume={47},
      number={3-4},
       pages={329\ndash 407},
      review={\MR{1876932}},
}

\bib{GrSea94}{article}{
      author={Grove, Karsten},
      author={Searle, Catherine},
       title={Positively curved manifolds with maximal symmetry-rank},
        date={1994},
     journal={J. Pure Appl. Algebra},
      volume={91},
      number={1-3},
       pages={137\ndash 142},
         url={https://doi.org/10.1016/0022-4049(94)90138-4},
      review={\MR{1255926}},
}

\bib{GrSea}{article}{
      author={Grove, Karsten},
      author={Searle, Catherine},
       title={Differential topological restrictions curvature and symmetry},
        date={1997},
     journal={J. Differential Geom.},
      volume={47},
      number={3},
       pages={530\ndash 559},
         url={http://projecteuclid.org/euclid.jdg/1214460549},
      review={\MR{1617636}},
}

\bib{GVZ}{article}{
      author={Grove, Karsten},
      author={Verdiani, Luigi},
      author={Ziller, Wolfgang},
       title={An exotic {$T_1\Bbb S^4$} with positive curvature},
        date={2011},
     journal={Geom. Funct. Anal.},
      volume={21},
      number={3},
       pages={499\ndash 524},
         url={https://doi.org/10.1007/s00039-011-0117-8},
      review={\MR{2810857}},
}

\bib{GW}{article}{
      author={Grove, Karsten},
      author={Wilking, Burkhard},
       title={A knot characterization and 1-connected nonnegatively curved
  4-manifolds with circle symmetry},
        date={2014},
     journal={Geom. Topol.},
      volume={18},
      number={5},
       pages={3091\ndash 3110},
         url={https://doi.org/10.2140/gt.2014.18.3091},
      review={\MR{3285230}},
}

\bib{GWZ}{article}{
      author={Grove, Karsten},
      author={Wilking, Burkhard},
      author={Ziller, Wolfgang},
       title={Positively curved cohomogeneity one manifolds and 3-{S}asakian
  geometry},
        date={2008},
     journal={J. Differential Geom.},
      volume={78},
      number={1},
       pages={33\ndash 111},
         url={http://projecteuclid.org/euclid.jdg/1197320603},
      review={\MR{2406265}},
}

\bib{GZ00}{article}{
      author={Grove, Karsten},
      author={Ziller, Wolfgang},
       title={Curvature and symmetry of {M}ilnor spheres},
        date={2000},
     journal={Ann. of Math. (2)},
      volume={152},
      number={1},
       pages={331\ndash 367},
         url={https://doi.org/10.2307/2661385},
      review={\MR{1792298}},
}

\bib{GZ02}{article}{
      author={Grove, Karsten},
      author={Ziller, Wolfgang},
       title={Cohomogeneity one manifolds with positive {R}icci curvature},
        date={2002},
     journal={Invent. Math.},
      volume={149},
      number={3},
       pages={619\ndash 646},
         url={https://doi.org/10.1007/s002220200225},
      review={\MR{1923478}},
}

\bib{HaSa}{article}{
      author={Haefliger, Andr\'{e}},
      author={Salem, \'{E}liane},
       title={Actions of tori on orbifolds},
        date={1991},
     journal={Ann. Global Anal. Geom.},
      volume={9},
      number={1},
       pages={37\ndash 59},
         url={https://doi.org/10.1007/BF02411354},
      review={\MR{1116630}},
}

\bib{HaSe}{article}{
      author={Harvey, John},
      author={Searle, Catherine},
       title={Orientation and symmetries of {A}lexandrov spaces with
  applications in positive curvature},
        date={2017},
     journal={J. Geom. Anal.},
      volume={27},
      number={2},
       pages={1636\ndash 1666},
         url={https://doi.org/10.1007/s12220-016-9734-7},
      review={\MR{3625167}},
}

\bib{HarSea}{article}{
      author={Harvey, John},
      author={Searle, Catherine},
       title={Positively curved alexandrov spaces with circle symmetry in
  dimension $4$},
        date={2018},
     journal={\href{https://arxiv.org/pdf/1805.09362.pdf}{arXiv:1805.09362}
  [math.DG]},
}

\bib{HarSea19}{article}{
      author={Harvey, John},
      author={Searle, Catherine},
       title={Almost non-negatively curved 4-manifolds with torus symmetry},
        date={2020},
     journal={Proc. Amer. Math. Soc.},
      volume={148},
      number={11},
       pages={4933\ndash 4950},
         url={https://doi.org/10.1090/proc/15093},
      review={\MR{4143405}},
}

\bib{Hatcher}{book}{
      author={Hatcher, Allen},
       title={Algebraic topology},
   publisher={Cambridge University Press, Cambridge},
        date={2002},
        ISBN={0-521-79160-X; 0-521-79540-0},
      review={\MR{1867354}},
}

\bib{Hoe}{thesis}{
      author={Hoelscher, Corey~{A}.},
       title={Classification of cohomogeneity one manifolds in low dimensions},
        type={Ph.D. Thesis},
        date={2007},
        note={available at
  \url{https://repository.upenn.edu/dissertations/AAI3260915}},
}

\bib{HsKl}{article}{
      author={Hsiang, Wu-Yi},
      author={Kleiner, Bruce},
       title={On the topology of positively curved {$4$}-manifolds with
  symmetry},
        date={1989},
     journal={J. Differential Geom.},
      volume={29},
      number={3},
       pages={615\ndash 621},
         url={http://projecteuclid.org/euclid.jdg/1214443064},
      review={\MR{992332}},
}

\bib{PerSeiSea}{article}{
      author={Jaramillo, M.},
      author={Perales, R.},
      author={Rajan, P.},
      author={Searle, C.},
      author={Siffert, A.},
       title={Alexandrov spaces with integral current structure},
        date={2021},
     journal={Comm. Anal. Geom.},
      volume={29},
      number={1},
       pages={115\ndash 149},
}

\bib{Kap}{incollection}{
      author={Kapovitch, Vitali},
       title={Perelman's stability theorem},
        date={2007},
   booktitle={Surveys in differential geometry. {V}ol. {XI}},
      series={Surv. Differ. Geom.},
      volume={11},
   publisher={Int. Press, Somerville, MA},
       pages={103\ndash 136},
         url={https://doi.org/10.4310/SDG.2006.v11.n1.a5},
      review={\MR{2408265}},
}

\bib{Kim}{article}{
      author={Kim, M.~Ho},
       title={Toral actions on {$4$}-manifolds and their classifications},
        date={1993},
     journal={Trans. Amer. Math. Soc.},
      volume={335},
      number={1},
       pages={105\ndash 130},
         url={https://doi.org/10.2307/2154260},
      review={\MR{1052907}},
}

\bib{Kle}{thesis}{
      author={Kleiner, Bruce},
       title={Riemannian four-manifolds with nonnegative curvature and
  continuous symmetry},
        type={Ph.D. Thesis},
        date={1990},
}

\bib{Mitsuishi}{article}{
      author={Mitsuishi, Ayato},
       title={Orientability and fundamental classes of alexandrov spaces with
  applications},
        date={2016},
     journal={\href{https://arxiv.org/pdf/1610.08024.pdf}{arXiv:1610.08024}
  [math.MG]},
}

\bib{Mos}{article}{
      author={Mostert, Paul~S.},
       title={On a compact {L}ie group acting on a manifold},
        date={1957},
     journal={Ann. of Math. (2)},
      volume={65},
       pages={447\ndash 455},
         url={https://doi.org/10.2307/1970056},
        note={Errata, ``On a compact Lie group acting on a manifold''. Ann. of
  Math. (2) 66 (1957), no. 3, 589},
      review={\MR{85460}},
}

\bib{Neu}{inproceedings}{
      author={Neumann, Walter~D.},
       title={{$3$}-dimensional {$G$}-manifolds with {$2$}-dimensional orbits},
        date={1968},
   booktitle={Proc. {C}onf. on {T}ransformation {G}roups ({N}ew {O}rleans,
  {L}a., 1967)},
   publisher={Springer, New York},
       pages={220\ndash 222},
      review={\MR{0245043}},
}

\bib{Nun}{article}{
      author={N\'{u}\~{n}ez Zimbr\'{o}n, Jes\'{u}s},
       title={Closed three-dimensional {A}lexandrov spaces with isometric
  circle actions},
        date={2018},
     journal={Tohoku Math. J. (2)},
      volume={70},
      number={2},
       pages={267\ndash 284},
         url={https://doi.org/10.2748/tmj/1527904822},
      review={\MR{3810241}},
}

\bib{OrlRayI}{article}{
      author={Orlik, Peter},
      author={Raymond, Frank},
       title={Actions of the torus on {$4$}-manifolds. {I}},
        date={1970},
     journal={Trans. Amer. Math. Soc.},
      volume={152},
       pages={531\ndash 559},
         url={https://doi.org/10.2307/1995586},
      review={\MR{268911}},
}

\bib{OrlRayII}{article}{
      author={Orlik, Peter},
      author={Raymond, Frank},
       title={Actions of the torus on {$4$}-manifolds. {II}},
        date={1974},
     journal={Topology},
      volume={13},
       pages={89\ndash 112},
         url={https://doi.org/10.1016/0040-9383(74)90001-9},
      review={\MR{348779}},
}

\bib{Pao1}{article}{
      author={Pao, Peter~Sie},
       title={The topological structure of {$4$}-manifolds with effective torus
  actions. {I}},
        date={1977},
     journal={Trans. Amer. Math. Soc.},
      volume={227},
       pages={279\ndash 317},
         url={https://doi.org/10.2307/1997462},
      review={\MR{431231}},
}

\bib{Pao2}{article}{
      author={Pao, Peter~Sie},
       title={The topological structure of {$4$}-manifolds with effective torus
  actions. {II}},
        date={1977},
     journal={Illinois J. Math.},
      volume={21},
      number={4},
         url={http://projecteuclid.org/euclid.ijm/1256048937},
      review={\MR{455000}},
}

\bib{Per2}{article}{
      author={Perel\cprime~man, G.~Ya.},
       title={Elements of {M}orse theory on {A}leksandrov spaces},
        date={1993},
     journal={Algebra i Analiz},
      volume={5},
      number={1},
       pages={232\ndash 241},
      review={\MR{1220498}},
}

\bib{Per}{article}{
      author={Perelman, G.},
       title={Alexandrov's spaces with curvatures bounded from below {II}},
        date={1991},
     journal={preprint},
        note={available at
  \url{https://anton-petrunin.github.io/papers/alexandrov/perelmanASWCBFB2+.pdf}},
}

\bib{Pet}{article}{
      author={Petrunin, A.},
       title={Parallel transportation for {A}lexandrov space with curvature
  bounded below},
        date={1998},
     journal={Geom. Funct. Anal.},
      volume={8},
      number={1},
       pages={123\ndash 148},
         url={https://doi.org/10.1007/s000390050050},
      review={\MR{1601854}},
}

\bib{SeaY}{article}{
      author={Searle, Catherine},
      author={Yang, DaGang},
       title={On the topology of non-negatively curved simply connected
  {$4$}-manifolds with continuous symmetry},
        date={1994},
     journal={Duke Math. J.},
      volume={74},
      number={2},
       pages={547\ndash 556},
         url={https://doi.org/10.1215/S0012-7094-94-07419-X},
      review={\MR{1272983}},
}

\bib{Wie}{article}{
      author={Wiemeler, Michael},
       title={Classification of rationally elliptic toric orbifolds},
        date={2020},
     journal={Arch. Math. (Basel)},
      volume={114},
      number={6},
       pages={641\ndash 647},
         url={https://doi.org/10.1007/s00013-019-01430-6},
      review={\MR{4102335}},
}

\bib{Wil}{article}{
      author={Wilking, Burkhard},
       title={Positively curved manifolds with symmetry},
        date={2006},
     journal={Ann. of Math. (2)},
      volume={163},
      number={2},
       pages={607\ndash 668},
         url={https://doi.org/10.4007/annals.2006.163.607},
      review={\MR{2199227}},
}

\bib{Yeroshkin}{article}{
      author={Yeroshkin, D.},
       title={On {G}eometry and {T}opology of 4-{O}rbifolds},
        date={2014},
     journal={\href{https://arxiv.org/pdf/1411.1700.pdf}{arXiv:1411.1700}
  [math.DG]},
}

\end{biblist}
\end{bibdiv}

\end{document}